\documentclass[11pt, reqno]{amsart}
\usepackage{amsthm, mathrsfs, amsmath, amstext, esint, amsxtra, amsfonts, slashed, dsfont, amssymb}
\usepackage{cite, relsize}
\usepackage[scale=0.75, centering, headheight=14pt]{geometry}
\usepackage[latin1]{inputenc}
\usepackage{microtype}
\usepackage{color,enumitem,graphicx,caption}
\usepackage{subcaption}
\usepackage[colorlinks=true,urlcolor=blue, citecolor=red,linkcolor=blue,
linktocpage,pdfpagelabels, bookmarksnumbered,bookmarksopen]{hyperref}
\usepackage[hyperpageref]{backref}
\usepackage[english]{babel}
\usepackage{amsrefs}

\numberwithin{equation}{section}

\newtheorem{theorem}{Theorem}[section]
\newtheorem{lemma}[theorem]{Lemma}
\newtheorem{proposition}[theorem]{Proposition}
\newtheorem{definition}[theorem]{Definition}
\newtheorem{corollary}[theorem]{Corollary}

\newtheorem{remark}[theorem]{Remark}
\newtheorem*{theorem*}{Conjecture}
\newtheorem{theo}{Theorem}

\newtheoremstyle{remarkstyle}
{}{}{}{ }{\bfseries}{.}{ }{\theoremname{#1}\theoremnumber{ #2}\theoremnote{ (#3)}}
\theoremstyle{remarkstyle}

\newcommand{\N}{\mathbb N}

\newcommand{\R}{\mathbb R}

\title[Ground state energy for NLS with competing nonlinearities]{Ground state energy threshold and blow-up for NLS with competing nonlinearities} 

\author[J. Bellazzini]{Jacopo Bellazzini}
\address[J. Bellazzini]{Dipartimento di Matematica, Universit\`a Degli Studi di Pisa, Largo Bruno Pontecorvo, 5, 56127, Pisa, Italy}
\email{jacopo.bellazzini@unipi.it}

\author[L. Forcella]{Luigi Forcella}
\address[L. Forcella]{Department of Mathematics, Heriot-Watt University, and The Maxwell Institute for the Mathematical Sciences, Edinburgh, EH14 4AS, United Kingdom}
\email{l.forcella@hw.ac.uk}

\author[V. Georgiev]{Vladimir Georgiev}
\address[V. Georgiev]{Dipartimento di Matematica, Universit\`a Degli Studi di Pisa, Largo Bruno Pontecorvo, 5, 56127, Pisa, Italy, and Faculty of Science and Engineering, Waseda University, 3-4-1, Okubo, Shinjuku-ku, Tokyo 169-8555, Japan, and IMICBAS, Acad. Georgi Bonchev Str., Block 8, 1113 Sofia, Bulgaria}
\email{georgiev@dm.unipi.it}


\begin{document}
\begin{abstract}
We consider the nonlinear Schr\"odinger equation with combined nonlinearities, where the leading term is an intracritical focusing power-type nonlinearity, and the perturbation is given by a  power-type defocusing one. We completely answer the question wether the ground state energy, which is a threshold between global existence and formation of singularities, is achieved. For any prescribed mass, for mass-supercritical or mass-critical defocusing perturbations, the ground state energy is achieved by a radially symmetric and decreasing solution to the associated stationary equation. For mass-subcritical perturbations, we show the existence of a critical prescribed mass, precisely the mass of the unique, static, positive solution to the associated elliptic equation, such that the ground state energy is achieved for any mass equal or smaller than the critical one. Moreover, the ground state energy is not achieved for mass larger than the critical one. As a byproduct of the variational characterization of the ground state energy, we prove the existence of blowing-up solutions in finite time, for any energy below the ground state energy threshold.\end{abstract}
	
\maketitle

\section{Introduction}
In this paper, we consider the following Schr\"odinger equation with competing nonlinearities: 
\begin{equation}\label{cNLS}
i\partial_{t}u+\Delta u=\lambda_1|u|^{q-1}u+\lambda_2|u|^{p-1}u, \quad (t,x)\in\R\times\R^d.
\end{equation}
Equation \eqref{cNLS} is a nonlinear Schr\"odinger equation with two nonlinearities of pure-power-type, which arises in the description of many physical models, determined by different configurations of the parameters $d,p,q, \lambda_1, \lambda_2$ appearing above. See \cite{GFTC, KA, Malo, PAK} and references therein for physical insights and justifications of the models. From a mathematical point of view, after the pioneering works by Zhang \cite{Zhang}, Tao, Vi\c san, and Zhang \cite{TVZ}, and Miao, Xu, and Zhao \cite{MXZCMP}, there has been an always increasing interest in NLS equations with combined nonlinearities. We refer the reader to \cite{Cheng, CaSp, LRN, KOPV, SJFA,SJDE} for extended discussions on the nowadays known results for \eqref{cNLS} for different choices of the structural constants $d,p,q,\lambda_1, \lambda_2.$ \\

Along this paper, we are interested in the three dimensional problem $d=3$,  and in exponents satisfying $1<q<p$, while the coefficients $\lambda_{1,2}$ are two real parameters with sign, and in particular $\lambda_1>0$ and $\lambda_2<0.$ Namely, we are concerned with  equation \eqref{cNLS} when the leading order nonlinearity is focusing, and the lower order nonlinearity is defocusing. Moreover, we are primarily interested in a leading order nonlinear term satisfying the following mass/energy intracriticality condition:
\[
p_{mc}:=\frac73<p<5:=p_{ec}, 
\]
where $p_{mc}=\frac73$ represents the mass-critical exponent, and $p_{ec}=5$ stands for the energy-critical exponent. We are therefore treating mass/energy intracritical nonlinearities for the focusing term, and we allow any mass-supercritical, mass-subcritical, or mass-critical perturbation for the defocusing term, subject to the condition that $q<p$. Without loss of generality, we fix the parameters $\lambda_{1,2}$ as being normalized as $\lambda_1=-\lambda_2=1;$ hence, for $(t,x)\in\R\times\R^3,$ the studied equation is
\begin{equation}\label{normNLS}
i\partial_{t}u+\Delta u=|u|^{q-1}u-|u|^{p-1}u,  \quad 1<q<p, \quad p\in\left(\frac73,5\right).
\end{equation}

As we also aim to deal with the dynamics for solutions to the evolution equation \eqref{normNLS}, we consider its associated Cauchy problem in the energy space. Hence, the initial value  problem reads 
\begin{equation}\label{NLS}
\left\{ \begin{aligned}
i\partial_{t}u+\Delta u&=|u|^{q-1}u-|u|^{p-1}u\\ 
u(0,x)&=u_0(x)\in H^1(\R^3)
\end{aligned}\right..
\end{equation}

\noindent Local existence of solutions to \eqref{NLS} is well-known, and we refer the reader to the classical monographs \cite{TC, Tao}. 
It is also well-known that solutions to \eqref{NLS} conserve along the flow the mass and the energy, the latter being defined by
\begin{equation}\label{def:en}
E(u(t)):=\frac{1}{2}\int_{\mathbb R^3}|\nabla u(t)|^2\,dx+\frac{1}{q+1}\int_{\mathbb R^3}|u(t)|^{q+1}\,dx-\frac{1}{p+1}\int_{\mathbb R^3}|u(t)|^{p+1}\,dx,
\end{equation}
the former being given by 
\begin{equation*}
M(u(t)):=\int_{\mathbb R^3}|u(t)|^2\,dx.
\end{equation*}
Conservation means  that for any $t\in(-T_{min},T_{max}),$ where $T_{min},T_{max}\in(0,\infty]$ are the minimal and maximal time of existence of the solution, respectively,  
\begin{equation*}\label{energy}
E(t):=E(u(t))=E(u(0))=E(u_0)
\end{equation*}
and  
\begin{equation*}\label{cons:mass}
M(t):=M(u(t))=M(u(0))=M(u_0).
\end{equation*}

\noindent For later purposes, let us write
\begin{equation*}\label{eq:kin-pot}
K(f)=\int_{\mathbb R^3}|\nabla f|^{2}dx,\qquad N_q(f)=\int_{\mathbb R^3}|f|^{q+1}dx, \qquad
N_p(f)=\int_{\mathbb R^3}|f|^{p+1}\,dx.
\end{equation*}
which are, up to the constants appearing in \eqref{def:en}, the kinetic energy functional, the component of the potential energy functional associated to the lower order nonlinear term, and the one associated to the leading order nonlinear term of \eqref{NLS}, respectively. With this terminology at hand, we can rewrite the energy functional as
\[
E(f)=\frac12K(f)+\frac{1}{q+1}N_q(f)-\frac{1}{p+1}N_p(f).
\]
We introduce the following quantity (the so-called Pohozaev functional)
\begin{equation}\label{eq:G}
G(f)=K(f)+\frac{3}{2}\left(\frac{q-1}{q+1}\right)N_q(f)-\frac{3}{2}\left(\frac{p-1}{p+1}\right)N_p(f),
\end{equation}
which will be crucial for the characterization of the dynamics for solutions to \eqref{NLS}, and we give the following useful identity:
\begin{equation}\label{eq:egkn}
E(f)-\frac23 \left(\frac{1}{p-1}\right)G(f)=\frac16\left(\frac{3p-7}{p-1}\right)K(f)+\left(\frac{p-q}{(q+1)(p-1)}\right)N_q(f).
\end{equation}
We note in particular that solutions to  \eqref{NLS} of the form $\psi(x,t)=e^{i \omega t}u(x),$ i.e. standing wave solutions,  fulfil $G(u)=0$.\\

By the fact that the mass is preserved during the time evolution,  we will study the energy functional under the mass constraint $\|u\|_{L^2(\R^3)}^2=\rho^2$. Specifically, we aim to study the existence of a least energy critical point of the energy with prescribed mass.
Notice that due to the focusing mass-supercritical nonlinearity, the energy functional is \emph{not} bounded from below under the mass constraint.
This fact will imply the existence of solutions to \eqref{NLS} that blow-up in finite time, namely $T_{min}<\infty$ and/or $T_{max}<\infty$. Hence, we introduce the following notation for $H^1(\R^3)$ functions having prescribed mass:
\[
S(\rho^2)=\left\{ f \in H^1(\R^3) \hbox{ such that } \|f\|_{L^2(\R^3)}^2=\rho^2\right\}.
\]

We recall the definition  of ground state for functional that are, as in our case, not bounded from below.
\begin{definition}\label{ground-state}
Let $\rho>0$ be arbitrary, we say that $u_{\rho} \in S(\rho^2)$ is a ground state if
\[
E(u_{\rho}) = \inf \left\{E(u) \hbox{ such that }  u \in S(\rho^2) \hbox{ and } E'|_{S(\rho^2)}(u) =0\right\}=\inf_{V(\rho^2)} E,
\]
where the set $V(\rho^2)$ is defined by
\[
V(\rho^2)=\left\{u\in H^1(\mathbb R^3) \hbox{ such that }  \|u\|_{L^2(\mathbb R^3)}^2=\rho^2  \hbox{ and } G(u)=0\right\}.
\]
\end{definition}

In other words, a function $u \in S(\rho^2)$ is called ground state if it minimizes the energy functional $E(u)$ among all the functions which belongs to $V(\rho^2)$ and, in particular, among all standing waves solution of \eqref{NLS}. 
Associated  to the ground state, there is a Lagrange multiplier $\omega \in \R,$ such that the couple $(u, \omega)$ solves
\begin{equation}\label{cNLS2}
-\Delta u+\omega u+|u|^{q-1}u-|u|^{p-1}u=0.
\end{equation}
Observe that, provided $u(x)$ solves \eqref{cNLS2}, the function $\psi(x,t):=e^{i \omega t}u(x)$ is a standing wave solution to \eqref{normNLS}.\\

Given $\rho>0$, a strategy to show that the infimum of the constrained energy is achieved by a ground state, hence the minimization problem actually admits  a minimizer, is to show that the energy functional $E$ defined in \eqref{def:en} has a mountain pass geometry on $S(\rho^2)$ (in the spirit of  \cite{JJ} for constrained minimization problems), and there exists a critical point at the mountain pass energy level. Precisely, we aim to prove  that there exists  a positive parameter $k>0$ such that,  given the set of paths 
\begin{equation}\label{Gamma2}
\Gamma(\rho^2) =\left\{g \in C([0,1];S(\rho^2))  \hbox{ such that } g(0) \in A_{k},E(g(1))<0\right\},
\end{equation}
where 
\[
A_{k}= \left\{f \in S(\rho^2) \hbox{ such that }  \|f\|_X^2\leq k\right\},
\]
and the  space $X$ is the Banach space of function having finite $X$-norm, the latter being defined by 
\begin{equation}\label{spaceX}
\|f\|_X:=K(f)^{\frac 12}+N_q(f)^{\frac{1}{q+1}}=\|f\|_{\dot H^1(\R^3)}+\|f\|_{L^{q+1}(\R^3)},
\end{equation}
then the following critical energy 
\[
I(\rho^2) := \inf_{g \in \Gamma(\rho^2)} \max_{t\in [0,1]}E(g(t))
\]
satisfies
\begin{equation*}\label{gamma}
I(\rho^2)> \max \left\{\max_{g \in \Gamma(\rho^2)}E(g(0)), \max_{g \in \Gamma(\rho^2)}E(g(1))\right\}.
\end{equation*}
It it standard, see \cite{AM}, that the mountain pass geometry induces the existence of a Palais-Smale sequence at the level $I(\rho^2)$. Namely a sequence $\{u_n\}_n \subset S(\rho^2)$ such that, as $n\to\infty,$
\[
E(u_n)=I(\rho^2)+o_n(1), \qquad \|E'|_{S(\rho^2)}(u_n)\|_{H^{-1}(\R^3)}=o_n(1).
\]
If one can show, in addition, the compactness of $\{u_n\}_n$, namely that, possibly up to a subsequence,  $u_n \to u$ in $H^1(\R^3)$, then a critical point is found at the level $I(\rho^2)$. As it will be shown later the mountain pass energy level  $I(\rho^2)$ coincides with $\inf_{V(\rho^2)} E$ and hence a Mountain pass critical point of the constrained energy corresponds to a Ground state solution. \\

For the understanding of the qualitative properties of $I(\rho^2),$ and for the existence result of mountain pass solutions with mass $\rho^2,$ besides \eqref{cNLS2}, it will be also crucial to consider the elliptic equation
\begin{equation}\label{EE}
-\Delta u+|u|^{q-1}u-|u|^{p-1}u=0,
\end{equation}
in $\dot H^1(\mathbb R^3)\cap L^{q+1}(\mathbb R^3)$. 
Equation \eqref{EE} is a specific case of the so called \emph{zero mass case} equation studied by Berestycki and Lions, see  \cite{BL}, and solutions to \eqref{EE} correspond to  \emph{static} solutions to \eqref{normNLS}, i.e. solutions of the equation \eqref{normNLS} which are independent of time.
\\

Before stating our main achievements, we  summarize -- for $1<q<p$, $\frac 73<p<5$ -- the known properties of $I(\rho^2),$ and those of positive solutions to \eqref{EE}, that will be crucial in our first theorem:
\begin{itemize}
\item the function $\rho\mapsto I(\rho^2)$ is continuous in $(0, \infty);$
\item $I(\rho^2) =  \inf_{V(\rho^2)} E$, i.e. the mountain pass energy is the least energy level for standing states, and therefore the ground state energy;
\item a positive, radially symmetric solution $U$ to \eqref{EE} do exists, see \cite{BL}. Furthermore, the positive radially symmetric solution that belongs to $\dot H^{1}(\mathbb R^3)\cap L^{q+1}(\mathbb R^3)$ is unique, see \cite{KZ, LN}, and fulfils  the following precise asymptotic behavior at infinity, see \cite{DS,DSW}:
\begin{equation}\label{eq:decay}
|U(r)|\sim \frac{1}{r^{\alpha}} \quad\hbox{ with }\quad \alpha=\max\left\{\frac{2}{q-1}, 1\right\}.
\end{equation}
\item for small value of $\rho$, the mountain pass energy $I(\rho^2)$ is attained by a function $u_{\rho}\in S(\rho^2)$, with an associated  Lagrange multiplier $\omega_\rho$,  which solves
\begin{equation}\label{eq:giap}
-\Delta u_{\rho}+\omega_{\rho} u_{\rho} +|u_{\rho}|^{q-1}u_{\rho}-|u_{\rho}|^{p-1}u_{\rho}=0.
\end{equation}
Moreover, $u_\rho$ is positive and radially symmetric, see \cite{J20, SJFA,SJDE};
\item for all $\omega\geq 0$ , a unique (see \cite{ST}), positive, radially symmetric solution to \eqref{eq:giap} exists, by minimizing the action functional $S_{\omega}(u)=E(u)+\frac\omega2M(u)$ on the Nehari manifold, in the spirit of \cite{FO}. For $\omega>0$ this solution is strongly unstable
as well as when $\omega=0$ and $1<q<\frac 73$, see \cite{FH}.\\
\end{itemize}

Having the above picture in mind, we are now able to state the main theorem of the paper.
\begin{theo}\label{theorem:main1}
Let $p,q$ be such that $p\in\left(\frac73,5\right)$ and $1<q<p.$ For any fixed $\rho>0,$  the energy functional $E(u)$ has a {\it mountain pass geometry} on $S(\rho^2)$. Furthermore, we have the following characterization. \\
\noindent \textbf{Mass-subcritical perturbations.} For $1<q<\frac{7}{3}$, i.e. when the defocusing term is mass-subcritical, there exists a critical mass $\rho_c>0,$ such that for all $0<\rho\leq \rho_c$  there exists a couple $(u_{\rho}, \omega_{\rho})\in S(\rho^2)\times \R$ such that $e^{i \omega_{\rho} t}u_{\rho}$ is a solution to \eqref{normNLS} with $E(u_{\rho})=I(\rho^2)$. Moreover:
\begin{itemize}
\item[$(i)$] the solution $u_{\rho}$ is radially symmetric, positive and decreasing;
\item[$(ii)$] the function $I(\rho^2)$ is strictly decreasing in $(0, \rho_c],$ with $\displaystyle\lim_{\rho\to 0}I(\rho^2)=+\infty,$ and constant in $[\rho_c, \infty)$; 
\item[$(iii)$] at the critical mass $\rho_c,$ the solution $u_{\rho_c}$  is the unique positive solution $U$ to the static equation \eqref{EE};
\item[$(iv)$] finally, $I(\rho^2)$ is not achieved for $\rho>\rho_c$.
\end{itemize}
\noindent \textbf{Mass-critical or supercritical perturbations.} If $\frac73 \leq q<p$, i.e. when the defocusing term is mass-critical or mass-supercritical, then 
for all $\rho$ there exists a couple $(u_{\rho}, \omega_{\rho})\in S(\rho^2)\times \R$ such that $e^{i \omega_{\rho} t}u_{\rho}$ is a solution to \eqref{normNLS} with 
 $E(u_{\rho})=I(\rho^2)$. Moreover:
\begin{itemize}
\item[$(i)$] the solution $u_{\rho}$ is radially symmetric, positive and decreasing; 
\item[$(ii)$] the function $I(\rho^2)$ is strictly decreasing in $(0, \infty)$ with  $\displaystyle \lim_{\rho\to 0}I(\rho^2)=+\infty$ and $\displaystyle\lim_{\rho\to \infty}I(\rho^2)=E(U)$.
\end{itemize}
\end{theo}

\begin{figure}[h!]
  \centering
  \begin{subfigure}[b]{0.49\linewidth}
    \includegraphics[width=\linewidth]{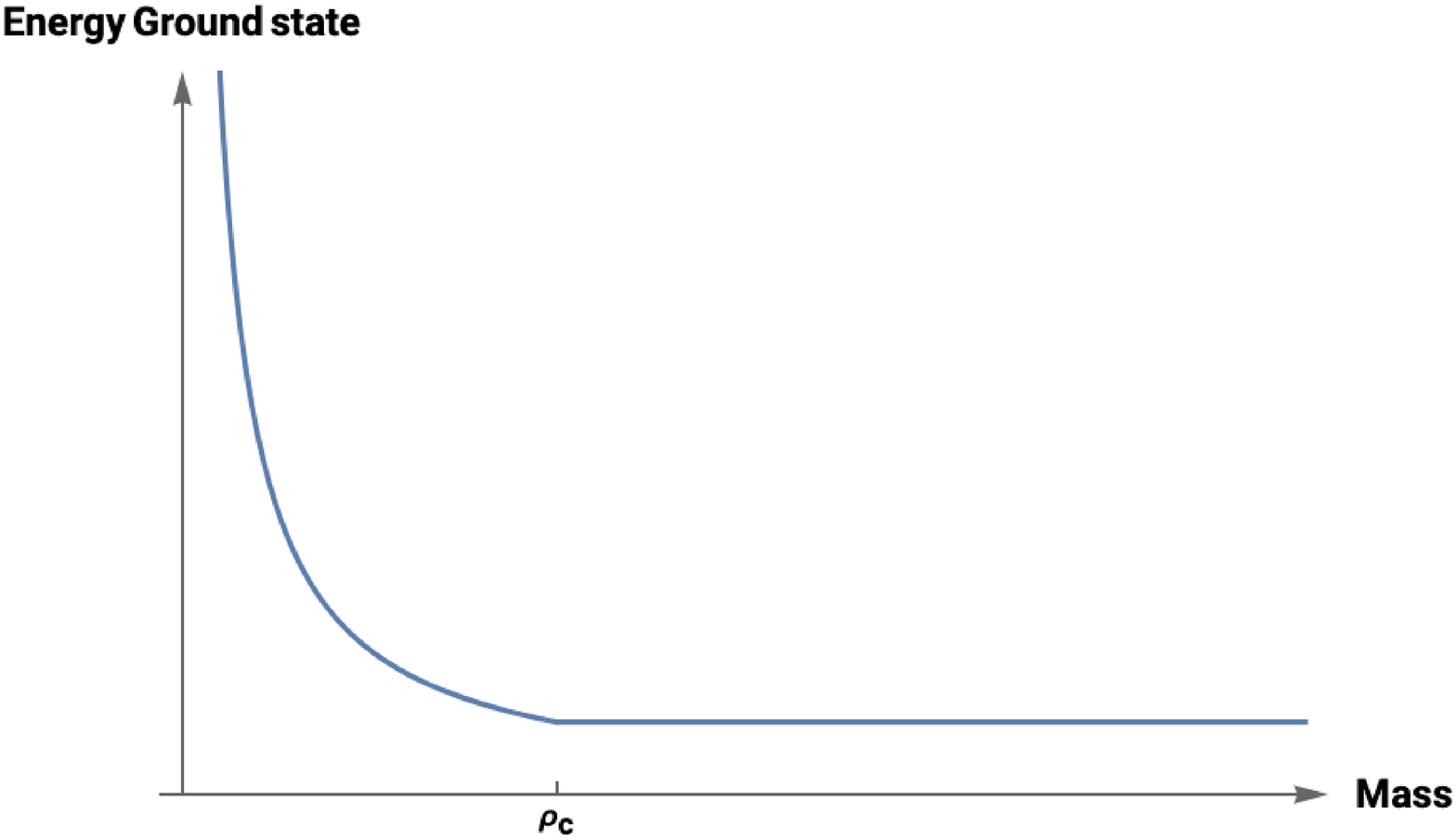}

  \end{subfigure}
  \begin{subfigure}[b]{0.49\linewidth}
    \includegraphics[width=\linewidth]{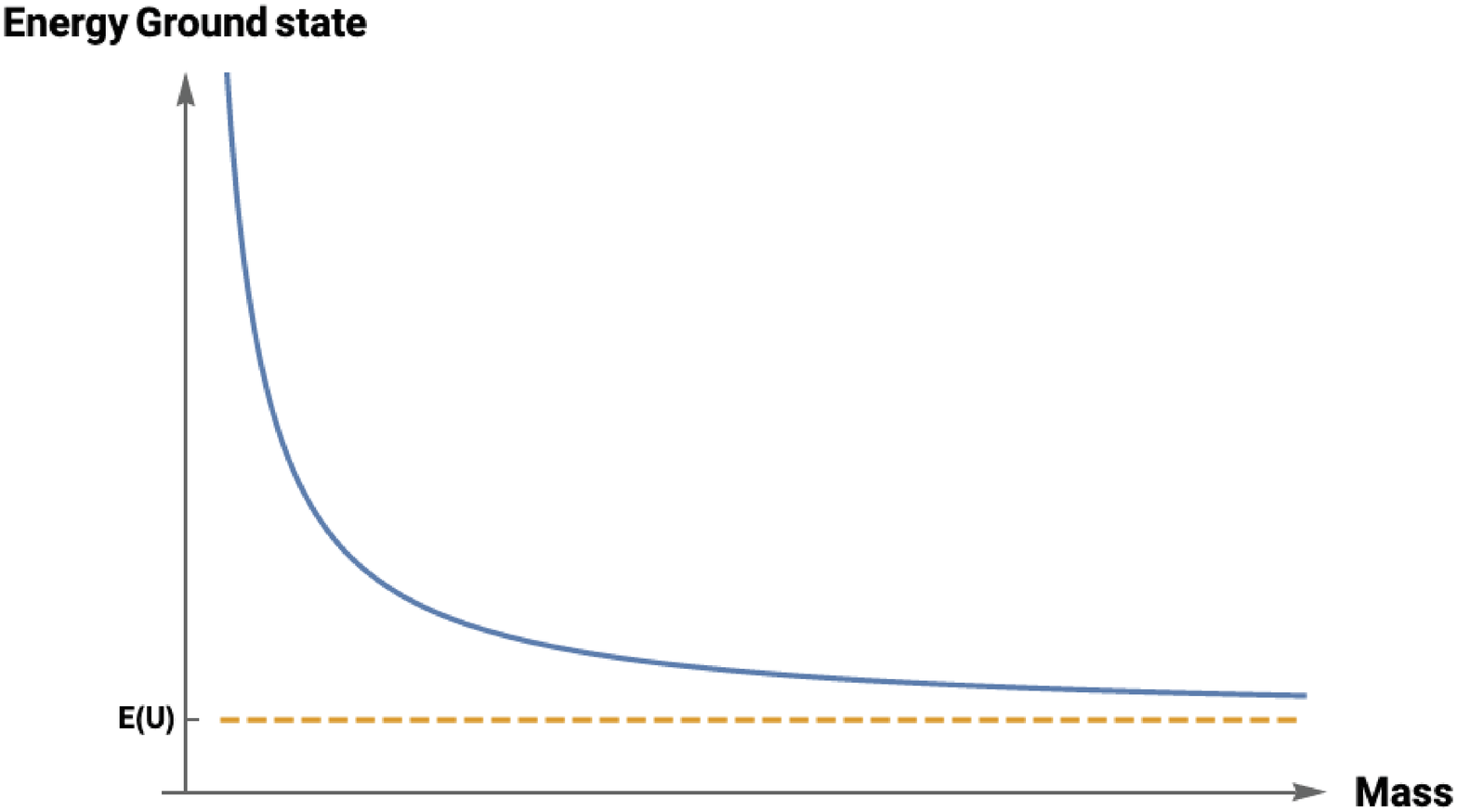}

  \end{subfigure}

  \caption{ Ground state energy as a function of the mass. \texttt{Left.} Mass-subcritical defocusing nonlinearity: the ground state energy is decreasing up to the threshold $\rho_c$,  given by the mass of the positive static solution, and then constant. \texttt{Right.} Mass-critical or mass-supercritical defocusing nonlinearity: the ground state energy is decreasing for all $\rho,$ approaching at infinity the energy of the unique positive static solution.}
 \label{fig:graph}
\end{figure}

Let us now comment on the achievements of Theorem \ref{theorem:main1}. A classical strategy to prove the existence of constrained critical point for mass-supercritical nonlinearities, is to show that $I(\rho^2)$ is a strictly decreasing function. When the latter monotonicity property is not easy to prove, a way to overcame this difficulty is to show that  weak limit  solves
\[
-\Delta u+\omega u +|u|^{p-1}u-|u|^{q-1}u=0
\]
with the corresponding Lagrange multiplier being  positive, see \cite{J20, SJFA,SJDE}. The positivity of the Lagrange multiplier proves that weak convergence to a critical point is in fact strong. For small $\rho$, due to functional inequalities, this positivity property is fulfilled.\\

Our approach is different, and it is not focused on proving that the Lagrange multiplier is positive (which is just a sufficient condition). Our idea  is related to the argument to prove strict sub-additivity for mass-subcritical nonlinearities in the spirit of \cite{BS}.
Given an arbitrary interval $J=(0,\rho]$, we prove in Lemma \ref{lem:necessary} that the existence of a positive radial solution to \eqref{EE}  with mass $\rho_0\in (0, \rho)$, is a \emph{necessary condition} for having \emph{not strict monotonicity} of the
function $I(\rho^2)$. The decay properties of the positive solution to \eqref{EE} given by \eqref{eq:decay} yield informations on the monotonicity of $I(\rho^2)$.
In the  case when  the nonlinearity $|u|^{p-1}u$ is focusing, and mass-supercritical and energy-subcritical, and the nonlinearity $|u|^{q-1}u$ is defocusing and mass-subcritical,  the unique positive solution to \eqref{EE}
belong to $H^1(\R^3),$ and hence the mass of this solution is well defined: $\rho_0=\rho_c$. As a byproduct, we show that  $I(\rho^2)$ is strictly decreasing  as soon as $0<\rho\leq \rho_c$ and constant for $\rho\geq \rho_c$. Indeed, assuming  on  the contrary that $I(\rho^2)$ is not strictly decreasing in $0<\rho\leq \rho_c,$ we deduce the existence of a  solution of \eqref{EE} with mass smaller than $\rho_c$.\\

On the other hand,  when the defocusing nonlinearity $|u|^{q-1}u$ is  mass-critical or mass-supercritical, the unique positive solution to \eqref{EE} does not belong to $H^1(\R^3)$
and hence $I(\rho^2)$ should be decreasing in $0<\rho<\infty$. When $\rho \to \infty,$ the ground state energy converges to the energy of the static solution of \eqref{EE}.
\begin{remark}
It is interesting to notice that in the case when both the nonlinearities are focusing and mass-supercritical, our argument, together with the decay estimate for
the positive solution to zero mass equation 
\begin{equation}\label{eq:FL}
-\Delta u=u^p+u^q
\end{equation}
due to Flucher and M\"uller, see \cite{FM}, permits to conclude that  $I(\rho^2)$ is strictly decreasing for all $\rho\in(0, \infty),$ thus $I(\rho^2)$
is achieved for all $\rho$. Indeed, positive and radially symmetric solutions to \eqref{eq:FL} have a lower bound decay given by $\frac{1}{|x|},$ and this fact
permits to exclude that a weak limit of the $L^2$-constrained  problem could solve \eqref{eq:FL}. With a different approach, in a more general setting, this case has been recently studied in \cite{J20}.
\end{remark}
\begin{remark}
The normalization of the coefficients in front of the two nonlinearities in \eqref{normNLS} do not affect the validity of the main Theorem. Indeed, for general coefficients with sign $\lambda_1>0$ and $\lambda_2<0,$ the achievements in Theorem \ref{theorem:main1} remain valid. 
\begin{remark}
We underline that our approach easily generalizes to arbitrary space dimensions. Indeed, considering again an intracritical focusing
nonlinearity and a defocusing lower order perturbation, the decay for the static solution is explicitly known to be
\end{remark}
\[
|U(r)|\sim \frac{1}{r^{\alpha}} \quad\hbox{ with }\quad \alpha=\max\left\{\frac{2}{q-1}, N-2\right\}.
\]
Following verbatim our approach, we get for instance that for $N\geq 5$, the static solution $U$ is in $L^2$ for all $q$, and hence the  ground state problem behaves as in \textsc{Figure \ref{fig:graph}}, \emph{\texttt{Left}}.
\end{remark}

With the results contained in Theorem \ref{theorem:main1} at hand, we can pass to the dynamical properties of (local) solutions to the time dependent problem \eqref{NLS}.  Specifically,  the ground state energy defines a threshold between global well-posedness and formation of singularities in finite time. Indeed, for an
initial datum $u_0\in H^1(\R^3)$ with mass $\rho_0^2=\|u_0\|_{L^2(\R^3)}^2,$ if the initial energy $E(u_0)<I(\|u_0\|_{L^2(\R^3)}^2)$ and $G(u_0)>0$,  the solution to \eqref{NLS} exists globally in time, i.e. $T_{min}=T_{max}=+\infty.$
This fact follows from the variational characterization of $I(\|u_0\|_{L^2(\R^3)}^2)$ together with \eqref{eq:egkn}.\\

For the blow-up part, some further assumptions on the symmetry of the solutions must be introduced. In what follows, we denote by $\Sigma_{3}$ the space of function with cylindrical symmetry and finite variance in the third direction. More precisely,
\begin{equation}\label{spaceSigma}
\Sigma_3=\{f\in H^1(\R^3) \hbox{ such that } f(x)=f(\bar x,x_3)=f(|\bar x|,x_3), \,f\in L^2(\R^3,x_3^2dx)\},
\end{equation}
where $\bar x=(x_1,x_2)$ and $|\bar x|=(x_1^2+x_2^2)^{1/2}.$

\begin{theo}\label{thm:blowup}
Let us consider an initial datum $u_0\in H^1(\R^3)$ with mass $\rho_0^2=\|u_0\|_{L^2(\R^3)}^2.$ Suppose that $E(u_0)<I(\|u_0\|_{L^2(\R^3)}^2)$ and $G(u_0)<0.$ Provided  one of the following conditions is satisfied, then we have finite time blow-up of the solution $u=u(t,x)$ to \eqref{NLS}, namely $T_{min},T_{max}<\infty:$ 
\begin{itemize}
\item[(i)] $u_0\in L^2(\R^3, |x|^2\,dx);$
\item[(ii)] $u$ is radial;
\item[(iii)] $\frac73 <p\leq 3,$ $1<q<p,$ and $u\in\Sigma_3.$
\end{itemize}
\end{theo}

We point-out that Theorem \ref{theorem:main1} is closely related to the  above Theorems \ref{thm:blowup} by the fact that, given an initial datum $u_0,$ for $\rho^2=\|u_0\|_{L^2(\R^3)}^2,$ the mountain pass energy level $I(\rho^2)$ gives an energy threshold to get global well-posedness or finite time blow-up for the solution to \eqref{NLS}, provided $G(u_0)$ has sign, and at least  one of the conditions $(i)$-$(iii)$ holds true for the blow-up part. It is worth mentioning that, once fixed the power of the focusing term in \eqref{NLS}, the qualitative behavior of the energy threshold $I(\rho^2)$ is only determined  by the property of the defocusing term to be mass-subcritical or not.
\begin{remark} We mention that $(ii)$-$(iii)$ of Theorem \ref{thm:blowup}, extend the blow-up results in \cite{Cheng, FH}. 
\end{remark}

We conclude this introduction by listing some notation used along the paper.
\subsection{Notations.} The $L^r(\R^3)$ spaces, with $r\in[1,\infty],$  denote the usual Lebesgue spaces, and $H^1(\R^3)$ stands for the usual Sobolev space of functions in $L^2(\R^3)$ with $\nabla u\in L^2(\R^3).$ With $\dot H^1(\R^3)$ we refer to its homogeneous version, and by $H^{-1}(\R^3)$ we denote the dual of $H^1(\R^3)$. As we work in the three-dimensional space, we drop the notation $\R^3,$ unless when needed to avoid confusion. For an analogous reason, we simply  write $\int\cdot\, dx$ when the domain of integration is the whole space. The space $X,$ introduced in \eqref{spaceX}, denotes the Banach space endowed with the norm $\|\cdot\|_X=\|\cdot\|_{\dot H^1}+\|\cdot\|_{L^{q+1}},$ while the space $\Sigma_3$ denotes the subspace of $H^1$ consisting of functions which are radial with respect to the $(x_1,x_2)$ components of $\R^3,$ and with finite variance with respect to the $x_3$-direction, see \eqref{spaceSigma}.

\section{Variational structure of the energy functional}\label{sec:3}

We analyze the geometry of the energy functional $E(u)$ constrained to $S(\rho^2),$ and to this aim we introduce the $L^2$-preserving  scaling
\begin{equation}\label{eq:scaling}
u^\mu(x):=\mu^{3/2}u(\mu x), \quad \mu>0.
\end{equation}
We recall the definition of the set $V(\rho^2):$
\[
V(\rho^2)=\left\{u\in H^1\quad \hbox{ such that } \quad \|u\|_{L^2}^2=\rho^2  \hbox{ and } G(u)=0\right\},
\]
where $G$ is defined in \eqref{eq:G}. In order to prove the fundamental Proposition below, in particular Proposition \ref{lem:growth}, point $(iii)$, we will use the following results. Although they are elementary, they are not evident facts. Thus, we provide here a proof for sake of completeness.

\begin{lemma}\label{lemmino1}
Let $z\in (0, \infty)$ and $f=az^2+b z^{\alpha}-cz^{\beta}$ with $b>0$, $c>0,$  $0<\alpha<2$, and $\beta>2$. Then $ f^\prime(z)>zf^{\prime\prime}(z).$
\end{lemma}
\begin{proof}
By direct computations
\[
\begin{aligned}
f^{\prime}(z)-zf^{\prime\prime}(z)&=b\alpha z^{\alpha-1}-c \beta z^{\beta-1}-b\alpha(\alpha-1) z^{\alpha-1}+c\beta(\beta-1) z^{\beta-1}\\
&= b\alpha(2-\alpha)z^{\alpha-1}+c\beta(\beta-2)z^{\beta-1},
\end{aligned}
\]
and both terms in the right-hand side of the equation above are strictly positive.
\end{proof}
As a consequence of the above elementary lemma, by calling $z_0$ the (unique) maximum of the function $f,$ we have the following.
\begin{corollary}\label{co:lemmino1}
Under the hypothesis of Lemma \ref{lemmino1}, then 
$f$ is concave in $[z_0, \infty)$ and $f^{\prime\prime}(z_0)<0.$
\end{corollary}
\begin{lemma}\label{lemmino2}
Let $z\in (0, \infty)$ and $f=az^2+bz^{\alpha}-cz^{\beta}$ with $b>0,$ $c>0,$ $2 \leq \alpha<\beta $, and $\beta>2$. Then, by calling $z_0$ the (unique) maximum of the function $f$
\[
f^{\prime\prime}(z_0)<0.
\]
\end{lemma}
\begin{proof}
The case $\alpha=2$ it is straightforward. Let us take $\alpha>2$ and assume that $f^{\prime\prime}(z_0)=0$. By the Taylor expansion formula,  we get
\[
f(z)=f(z_0)+\frac 16 f^{\prime\prime\prime}(z_0)(z-z_0)^3+o(|z-z_0|^3).
\]
By the fact that $z_0$ is a maximum point, we get that $f^{\prime\prime\prime}(z_0)=0.$
By easy computations we have
\[
0=f^{\prime}(z_0)-z_0f^{\prime\prime}(z_0)=b\alpha(2-\alpha)z_0^{\alpha-1}+c\beta(\beta-2)z_0^{\beta-1},
\]
and
\[
z_0^2f^{\prime\prime\prime}(z_0)=b\alpha(\alpha-1)(\alpha-2)z_0^{\alpha-1}-c\beta(\beta-1)(\beta-2)z_0^{\beta-1}.
\]
From the first identity above we get
\[
z_0^{\beta-\alpha}=\frac{b\alpha(\alpha-2)}{c\beta(\beta-2)},
\]
while from the second
\[z_0^{\beta-\alpha}=\frac{b\alpha(\alpha-1)(\alpha-2)}{c\beta(\beta-1)(\beta-2)},
\]
and they imply a contradiction, as $\alpha\neq\beta.$
\end{proof}

\begin{proposition}\label{lem:growth}
Let $u$ belong to the manifold $S(\rho^2).$ Then the following properties are  satisfied: 
\begin{itemize}
\item[$(i)$] $\frac{d}{d \mu} E(u^{\mu})=\frac{1}{\mu}G(u^{\mu})$, $\forall \mu>0$;
\item[$(ii)$] there exists a unique $\tilde\mu=\tilde\mu(u)>0$, such that $u^{\tilde\mu} \in V(\rho^2)$;
\item[$(iii)$] the map $\mu \mapsto E(u^{\mu})$ is concave on $[\tilde\mu, \infty),$ and $\frac{d^2}{d \mu^2} E(u^{\tilde\mu})<0$;
\item[$(iv)$] $E(u^{\mu})<E(u^{\tilde\mu})$, for any $\mu>0$ and $\mu \neq \tilde\mu$;
\item[$(v)$] $\tilde\mu(u)<1$ if and only if $G(u)<0$;
\item[$(vi)$] $\tilde\mu(u)=1$ if and only if $G(u)=0$;
\item[$(vii)$] the functional $G$ satisfies 
\begin{equation*}
G(u^\mu)
\begin{cases}
 >0,\quad \forall\, \mu \in (0,\tilde\mu(u))\\
 <0, \quad \forall\, \mu\in (\tilde\mu(u),+\infty)
\end{cases}.
\end{equation*}
\end{itemize}
\end{proposition}
\begin{proof}
We prove the claims of the Proposition, by starting with the first point $(i)$. 

\noindent Let us introduce the rescaled quantities (under the rescaling \eqref{eq:scaling}) $E(u^\mu)$ and $G(u^\mu):$
\begin{equation}\label{eq:resc:E}
E(u^\mu)=\frac12\mu^2 K(u)+\frac{1}{q+1}\mu^{\frac 32(q-1)}N_q(u)-\frac{1}{p+1}\mu^{\frac 32(p-1)}N_p(u)
\end{equation}
and 
\begin{equation}\label{eq:resc:G}
G(u^\mu)=\mu^2 K(u)+\frac{3}{2}\left(\frac{q-1}{q+1}\right) \mu^{\frac 32(q-1)}N_q(f)-\frac{3}{2}\left(\frac{p-1}{p+1}\right)\mu^{\frac 32(p-1)}N_p(f).
\end{equation}
Direct computations show that
\begin{equation}\label{eq:der:E}
\frac{d}{d\mu}E(u^\mu)=\mu K(u)+\frac{3}{2}\left(\frac{q-1}{q+1}\right)\mu^{\frac 32(q-1)-1}N_q(u)-\frac{3}{2}\left(\frac{p-1}{p+1}\right)\mu^{\frac 32(p-1)-1}N_p(u)
\end{equation}
and combining the  two identities \eqref{eq:resc:G} and \eqref{eq:der:E} we see that
\begin{equation}\label{eq:ide}
\frac{d}{d\mu}E(u^\mu)=\frac1\mu G(u^\mu)
\end{equation}
hence $(i)$ is proved.\\

Let us prove $(ii)$. We recall that $p\in\left(\frac73,5\right).$ We distinguish two cases. 

\noindent\emph{Case I: $\frac73\leq q<p$}. We are therefore considering  the defocusing  nonlinearity  mass-critical or mass-supercritical. We rewrite \eqref{eq:resc:G} as
\[
G(u^\mu)=\mu^{2}\left( K(u)+\frac{3}{2}\left(\frac{q-1}{q+1}\right)\mu^{\frac 32(q-1)-2}N_q(u)-\frac{3}{2}\left(\frac{p-1}{p+1}\right)\mu^{\frac 32(p-1)-2}N_p(u) \right)=\mu^{2} h(\mu),
\]
where
\[
h(\mu)=K(u)+\frac{3}{2}\left(\frac{q-1}{q+1}\right)\mu^{\frac 32(q-1)-2}N_q(u)-\frac{3}{2}\left(\frac{p-1}{p+1}\right)\mu^{\frac 32(p-1)-2}N_p(u).
\]
The fact that there exists only one $\tilde\mu$ such that $h(\tilde\mu)=0,$ which in turn implies  $G(u^{\tilde \mu})=0,$ it is straightforward when $\frac 32(q-1)-2=0$, i.e. when the defocusing term is mass- critical. In all other cases, we observe that $h^\prime(\mu)$ admits only one (real) zero, say $\mu_0>0$, and $h(\mu)$ is increasing
on $[0, \mu_0],$ and decreasing on $[\mu_0,+\infty)$.
Indeed we have
\[
h^\prime(\mu)=c_1\mu^{\frac 32(q-1)-3}N_q(u)-c_2\mu^{\frac 32(p-1)-3}N_p(u) 
\]
for some $c_1,c_2>0,$ and the latter function  admits only one zero. This property,  combined with the fact that $h(\mu)\to-\infty$ as $\mu\to+\infty$, implies that there exists only one $\tilde\mu$ such that $h(\tilde\mu)=0,$ hence $G(u^{\tilde \mu})=0$. We recall
that the scaling \eqref{eq:scaling} preserves the $L^2$-norm, then it is straightforward that $u^\mu \in S(\rho^2)$.\\

\noindent \emph{Case II: $1<q<\frac73<p$}.  Here we are  dealing with the defocusing term in the mass-subcritical range. In this case it is convenient to rewrite \eqref{eq:resc:G} as
\[
G(u^\mu)=\mu^{\frac 32(q-1)}\left( \mu^{\frac{7-3q}{2}}K(u)+\frac{3}{2}\left(\frac{p-1}{p+1}\right)N_q(u)-\frac{3}{2}\left(\frac{p-1}{p+1}\right)\mu^{\frac 32(p-q)}N_p(u) \right)=\mu^{\frac 32(q-1)}h(\mu)
\]
where 
\[
h(\mu)=\mu^{\frac{7-3q}{2}}K(u)+\frac{3}{2}\left(\frac{p-1}{p+1}\right)N_q(u)-\frac{3}{2}\left(\frac{p-1}{p+1}\right)\mu^{\frac 32(p-q)}N_p(u).
\]
Arguing as before, we have that $h^\prime(\mu)$ admits only one (real) zero, and hence that there exists only one $\tilde\mu$ such that $h(\tilde\mu)=0,$ thus $G(u^{\tilde \mu})=0$. The proof is done, by recalling that the scaling is $L^2$-preserving.  The proof of $(ii)$ is complete.\\

Point $(iii)$ is proved by means of Lemma \ref{lemmino1} and Lemma \ref{lemmino2}. \\
\noindent We rewrite the rescaled energy \eqref{eq:resc:E} as
\begin{equation*}\label{eq:resc:E2}
\begin{aligned}
E(u^\mu)&=\frac12\mu^2 K(u)+\frac{1}{q+1}\mu^{\frac 32(q-1)}N_q(u)-\frac{1}{p+1}\mu^{\frac 32(p-1)}N_p(u)\\
&=a\mu^2+b\mu^{\frac 32(q-1)}-c\mu^{\frac 32(p-1)}
\end{aligned}
\end{equation*}
We distinguish two cases.\\

\noindent\emph{Case I: $q<\frac73$.} By observing that $\alpha=\frac 32(q-1)<2$ and $\beta=\frac 32(p-1)>2,$ we apply Corollary  \ref{co:lemmino1} and we conclude.\\
\noindent\emph{Case II: $q\geq\frac73$.}  In this case, by observing that $\alpha=\frac 32(q-1)\geq2$ and $\beta=\frac 32(p-1)>2,$ we use Lemma \ref{lemmino2} and we conclude.
\\

We show  point $(iv)$.

\noindent We  observe that the rescaled energy $E(u^\mu),$ defined in \eqref{eq:resc:E}, is positive and increasing in a neighbourhood of the origin, and its derivative $\frac{d}{d\mu}E(u^\mu)$ admits only one zero, which is a point of maximum for $E(u^\mu),$
and it is exactly the point $\tilde\mu,$ where for the latter fact we employ the point $(i)$ and \eqref{eq:ide}. This in turn gives $(iv)$.\\

Let us consider point $(v)$. \\
\noindent Assume $G(u)<0$ and suppose that $\tilde\mu>1.$ As $G(u^\mu)>0$ for any $ \mu< \tilde\mu,$ in particular we have $G(u)>0,$ which is a contradiction. Viceversa, assume that $\tilde\mu<1$; then for any $\mu>\tilde\mu$ the quantity $G(u^\mu)<0$, and in particular $G(u)<0$. \\

We can conclude the proof of the Proposition by observing that  point $(vi)$ is straightforward by definition of $\tilde\mu$, in conjunction with the previous discussions, while point $(vii)$ is clear from the definition of $\tilde\mu$ and the identity \eqref{eq:ide}, observing that the right-hand side of \eqref{eq:ide} is positive.
\end{proof}

As a simple consequence  of the previous Proposition and the definition of the rescaling \eqref{eq:scaling}, we have the following.
\begin{corollary}\label{lem:base}
Let $u \in S(\rho^2).$ Then we have:
\begin{itemize}
\item[$(i)$] $K(u^\mu)\to0$ and $N_q(u^\mu)\to0$  as $\mu\to0$;
\item[$(ii)$] $K(u^{\mu}) \to \infty,$  $N_q(u^{\mu}) \to \infty,$ and $E(u^{\mu})  \to -\infty$, as $\mu \to \infty$;
\item[$(iii)$] if $E(u)<0,$ then $G(u)<0.$
\end{itemize}
\end{corollary}
\begin{proof}
The first  point is trivial. The second point is a consequence of Proposition \ref{lem:growth}. The second one straightforwardly comes from the identity \eqref{eq:egkn}.
\end{proof}
Now, let us consider the  space $X=\dot H^{1}(\mathbb R^3)\cap L^{q+1}(\mathbb R^3),$ which is a reflexive Banach space
equipped with the norm 
\begin{equation}\label{Xagain}
\|u\|_X=K(u)^{\frac 12}+N_q(u)^{\frac{1}{q+1}},
\end{equation}
and we prove that the energy functional has a mountain pass geometry on $X.$
\begin{proposition}\label{prop:mp0}
$E(u)$ has a mountain pass geometry on $X$.
\end{proposition}
\begin{proof}
Let us first define the space of paths 
\begin{equation}\label{Gamma3}
\Gamma_X =\{g \in C([0,1];X)  \hbox{ such that } g(0)\in A_k,E(g(1))<0\}.
\end{equation}
where 
\[A_{k}= \{u \in X \hbox{ such that }  \left \|  u \right \|_{X}^2\leq k\}.
\]
In order to show the mountain pass geometry we aim to prove that
\begin{equation}\label{gamma2}
I := \inf_{g \in \Gamma_X} \max_{t\in [0,1]}E(g(t)) > \max \left\{\max_{g \in \Gamma_X}E(g(0)), \max_{g \in \Gamma_X}E(g(1))\right\}.
\end{equation}
Let us fix $\|u\|_{X}^2=b.$ From \eqref{Xagain}, by the Young inequality, we have $b\leq2\left(K(u)+N_q(u)^{\frac{2}{q+1}}\right)$. This fact implies that either 
\[
K(u)\geq \frac{b}{4}=\frac{1}{4}\|u\|_X^2
\] or 
\[
N_{q}(u)\geq \left(\frac{b}{4}\right)^{\frac{q+1}{2}}=\left(\frac{1}{4}\right)^{\frac{q+1}{2}}\|u\|_X^{q+1},
\]
which implies that for $u \in X$
\begin{equation}\label{eq:figo}
\frac 12 K(u)+\frac{1}{q+1}N_q(u)\geq \frac{1}{q+1}\left(K(u)+N_q(u)\right)\geq \frac{2}{q+1}\min \left\{\frac{1}{4}\|u\|_X^2,  \left(\frac{1}{4}\right)^{\frac{q+1}{2}}\|u\|_X^{q+1}\right\}.
\end{equation}
The Gagliardo-Nirenberg inequality yields  $N_{p}(u)\leq C\|u\|_X^{p+1},$ and hence
\begin{equation}\label{eq:vaazero}
\begin{aligned}
E(u)&\leq \frac{1}{2}K(u)+\frac{1}{q+1}N_q(u)+\frac{1}{p+1} N_p(u)\leq   \frac{1}{2}\|u\|_{X}^2+ \frac{1}{q+1}\|u\|_{X}^{q+1}+C \|u\|_{X}^{p+1}\\
E(u)&= \frac{1}{2}K(u)+\frac{1}{q+1}N_q(u)-\frac{1}{p+1} N_p(u)\geq   \frac{1}{2}\|u\|_{X}^2+ \frac{1}{q+1}\|u\|_{X}^{q+1}-C \|u\|_{X}^{p+1}
\end{aligned}
\end{equation}
Thus, for a fixed $b$, by defining the set 
\[
C_{b}= \left\{u \in X \hbox{ such that } \left \|  u \right \|_{X}^2=b\right\},
\]
we derive, by \eqref{eq:figo} and \eqref{eq:vaazero},  that for any $u \in C_{b}$
\[
E(u)\geq  \frac{2}{q+1}\min \left\{\frac{b}{4},  \left(\frac{b}{4}\right)^{\frac{q+1}{2}}\right\}-C b^{\frac{p+1}{2}},
\]
and, by the definition of $G$, see \eqref{eq:G},
\[
G(u)\geq 3\left(\frac{q-1}{q+1}\right)\min \left\{\frac{b}{4},  \left(\frac{b}{4}\right)^{\frac{q+1}{2}}\right\}-C b^{\frac{p+1}{2}}.
\]
As a byproduct, there exists $b_0\ll1$ such that, thanks to \eqref{eq:vaazero},  for $k\ll b_0$
\[
0< \sup_{u \in A_{k}} E(u)< \inf_{u \in C_{b_0}} E(u),
\] 
and  
$\inf_{u \in C_{b}} G(u)>0$ when $0<b<b_0.$
The proof of Proposition \ref{prop:mp0} follows directly from these two estimates. From the scaling introduced in \eqref{eq:scaling},
the class of paths defined in \eqref{Gamma3} fulfilling \eqref{gamma2} is not empty. Indeed, given $g\in \Gamma_X$ there exists $\bar t\in[0,1]$ such that $g(\bar t)\in C_{b_0}.$
\end{proof}
We now  introduce  the space  
\begin{equation}\label{eq:space-W}
W=\left\{u\in X   \hbox{ such that } G(u)=0\right\}.
\end{equation}
We have the following characterization of the mountain pass energy $I$ in $X.$
\begin{proposition} \label{prop:infe0} The minimization problem satisfies 
$I= \displaystyle \inf_{W} E$.
\end{proposition}
\begin{proof}
Let $v\in W$. Since  $G(v) =0,$  by considering the scaling 
$v^{\mu}(x)=\mu^{\frac 32}v(\mu x)$,  we deduce from Corollary \ref{lem:base} that there exists  $\mu_1\ll1$ and  $\mu_2\gg1$ such that $u^{\mu_1}\in A_{k}$ and $E(u^{\mu_2})<0$. Thus, if we define
\[
g(\lambda)=v^{(1-\lambda)\mu_1+\lambda \mu_2}, \quad\hbox{ for }\quad \lambda\in[0,1],
\]
we obtain a path in $\Gamma_X$. By the definition of $I$ 
\[
I\leq \max_{\lambda\in [0,1]}E(g(\lambda))=E(v),
\]
which implies that $I\leq \displaystyle \inf_{W}E.$
On the other hand, any path $g(t)$ in $\Gamma_X$ crosses $W$, by continuity. This shows that
\[
\max_{t \in [0,1]} E(g(t))\geq \inf_{u \in W} E(u),
\]
and hence that $I\geq \inf_{W}E$
\end{proof}

\begin{proposition}\label{prop:mp}
$E(u)$ has a mountain pass geometry on $S(\rho^2)$.
\end{proposition}
\begin{proof}
Analogous to the proof of Proposition \ref{prop:mp0}, with $\Gamma(\rho^2)$ defined in \eqref{Gamma2} replacing $\Gamma_X$.

\end{proof}

\begin{proposition} \label{prop:infe}
The minimization problem satisfies $I(\rho^2) = \displaystyle \inf_{V(\rho^2)} E$.
\end{proposition}
\begin{proof}
Similar to the proof of Proposition \ref{prop:infe0}, replacing $W$ with $V(\rho^2).$
\end{proof}

\section{Monotonicity of  $I(\rho^2)$ and the connection with the solution to \eqref{EE}}
As stated in the Introduction, the mountain pass geometry implies the existence of a Palais-Smale sequence $\{u_n\}_n \subset S(\rho^2)$ such that
\[
E(u_n)=I(\rho^2)+o_n(1) \quad \hbox{ and } \quad \|E'|_{S(\rho^2)}(u_n)\|_{H^{-1}}=o_n(1).
\]
Inspired by an argument due to \cite{Gh, GP},  we can strengthen this information and select a specific sequence localized around $V(\rho^2)$, namely such that
$\hbox{dist}(u_n, V(\rho^2))=o_n(1)$. 
Let us notice that  for any $n\in\N$  and any $w\in V(\rho^2),$ we may write, for $\theta\in[0,1]$
\[
G(u_n)=G(w)+d G(\theta u_n+(1-\theta)w)(u_n-w)=dG(\theta u_n+(1-\theta)w)(u_n-w).
\]
Therefore, by choosing a sequence  $\{w_m\}_n\subset V(\rho^2)$ such that
\[
\|u_n-w_m\|\to \hbox{dist}(u_n,V(\rho^2))
\]
as $m\to\infty,$ we get that $G(u_n)=o_n(1)$, as $\hbox{dist}(u_n,V(\rho^2))\to0$.
In particular we can consider a Palais-Smale sequence for which $G(u_n)=o_n(1).$\\

We first recall  the following known property. It basically says that  the monotonicity of $I(\rho^2)$ implies compactness.
\begin{lemma}\label{lem:mono}
Let $\{u_n\}_n \subset S(\rho^2)$ be such that
$E(u_n)=I(\rho^2)+o_n(1)$, $\|E'|_{S(\rho^2)}(u_n)\|_{H^{-1}}=o_n(1)$, $G(u_n)=o_n(1)$ as $n\to\infty,$ and $I(\rho^2)<I(\mu^2)$ for any $0<\mu<\rho.$ Then,
up to a possible space translation sequence, we have $u_n \to u$ in $H^1,$ with  $E(u)=I(\rho^2)$ and $G(u)=0.$
Moreover, the ground state solution $u$ is real, positive and radially symmetric. 
\end{lemma}
\begin{proof}
Firstly, we notice that there exists a  Palais-Smale sequence  $\{u_n\}_n \subset S(\rho^2)$ such that $E(u_n)\to I(\rho^2)$ and $G(u_n)=o_n(1)$ which is bounded in $X,$ as consequence of \eqref{eq:egkn}. \\
Let us first show that   $ u \neq 0$. We claim that $\|u_n\|_{L^{p+1}}\geq c>0.$ On the contrary, by assuming that $\|u_n\|_{L^{p+1}}=o_n(1),$ by interpolation
we  get $\|u_n\|_{L^{q+1}}=o_n(1),$ and by the fact that $G(u_n)=o_n(1)$ we  conclude that $K(u_n)=o_n(1)$.
On the other hand  
\[
E(u_n)-\frac23 \left(\frac{1}{p-1}\right)G(u_n)=\left(\frac{3p-7}{6p-6}\right)K(u_n)+\left(\frac{p-q}{(q+1)(p-1)}\right)N_q(u_n)=I(\rho^2)+o_n(1),
\]
and hence the latter contradicts the fact that $\|u_n\|_{L^{p+1}}=o_n(1)$.
At this point, since $\|u_n\|_2^2=\rho^2$ for any $n$, and $\sup_n\|u_n\|_{L^6}\leq C$, the classical pqr-Lemma, see \cite{FLL}, implies that  there exists a $\eta>0$, such that 
\begin{equation*}\label{hyp2}
\inf_n \left|\{ |u_n|>\eta \}\right|>0.
\end{equation*}
Here  $|\cdot |$ denote the Lebesgue measure of a measurable set. This fact, together with the Lieb Translation Lemma, see \cite{Lieb}, guarantees  the existence of a sequence $\{x_n\}_n\subset\mathbb R^3$ such that a subsequence of $\{u_n(x+ x_n)\}_n$ has a weak limit $ u \not\equiv 0$ in $H^1$. Now, let us prove the weak convergence is actually a strong convergence. Since  $ u$ is not trivial, and is a weak solution to \eqref{cNLS2},  we can assume 
that $u \in V(\mu^2)$ for some  $0<\mu\leq \rho$.
By the Brezis-Lieb Lemma, we have that 
\begin{equation}\label{splittings}
\begin{aligned}
K(u_n- u)+ K( u)&=K(u_n)+o_n(1), \\ 
N_q(u_n- u)+N_q( u)&=N_q(u_n)+o_n(1),  \\ 
N_p(u_n- u)+N_p( u)&=N_p(u_n)+o_n(1).
\end{aligned}
\end{equation}
Since $E(u_n) \to I(\rho^2),$ the splittings \eqref{splittings} above give
\begin{equation*}\label{101}
E(u_n- u)+E( u)=I(\rho^2)+o_n(1),
\end{equation*}
and we also have
\begin{equation}\label{1020}
G(u_n- u)+G( u)=G( u_n)+o_n(1).
\end{equation}
Since $ u\in V(\mu^2)$ we have that $E(u) \geq I(\mu^2),$  and we deduce therefore that
 \begin{equation*}\label{eq:mono}
 E(u_n - u)+I(\mu^2)\leq I(\rho^2)+o_n(1).
 \end{equation*}
Now, by using the relation
\[\left(\frac{3p-7}{6p-6}\right)K(u_n-u)+\left(\frac{p-q}{(q+1)(p-1)}\right)N_q(u_n-u)= E(u_n-u)-\frac23 \left(\frac{1}{p-1}\right)G(u_n-u),
\]
by the fact that $G(u_n-u)=o_n(1),$ thanks to \eqref{1020}
we obtain that
\[\left(\frac{3p-7}{6p-6}\right)K(u_n-u)+\left(\frac{p-q}{(q+1)(p-1)}\right)N_q(u_n-u)+I(\mu^2)\leq I(\rho^2)+o_n(1).
\]
By the monotonicity of $I(\rho^2)$
we necessarily deduce that  $\mu^2 = \rho^2$, $K(u_n- u)=o_n(1)$ and $N_q(u_n- u)=o_n(1)$. This proves the strong convergence of $\{u_n\}_n$ in $H^1$. The ground state solution $u$ can be selected real, positive and radially symmetric, by the well known facts concerning the  symmetric decreasing rearrangement, see \cite{BZ} for a reference.
\end{proof}
In what follows, we focus on the proof of some qualitative properties of the functional $I(\rho^2).$

\begin{lemma}
$I(\rho)$ diverges as $\rho\to0,$ i.e. $\lim_{\rho \to 0}  I(\rho^2)=+\infty$.
 \end{lemma}
\begin{proof}
In order to prove the Lemma, we make a connection with the usual NLS equation with one focusing nonlinearity of power-type. Indeed, to prove that the limit at $0^+$ of $I(\rho^2)$ do exists, 
we  first notice that $\lim_{\rho \to 0}\tilde I(\rho^2)=+\infty,$ where  $\tilde I, \tilde E,$ and $\tilde V$ are defined by
\[
\tilde I(\rho^2)=\inf_{u \in \tilde V(\rho^2)}\tilde E(u),
\]  
\[
\tilde E(u)=\frac12K(u)-\frac{1}{p+1}N_p(u),
\]
and
\[
\tilde V(\rho^2)=\left\{ u \in S(\rho^2) \hbox{ such that }\tilde G(u)=0  \right\},
\] 
respectively, where
\[
\tilde G(u)=K(u)-\frac{3}{2}\left(\frac{p-1}{p+1}\right)N_p(f).
\]  

Now, provided we take $\tilde u_{\rho}\in \tilde V(\rho^2)$ with $\tilde E(\tilde u_{\rho})=\tilde I(\rho^2),$ by the fact that $e^{i \omega t}\tilde u$ is a standing wave solution for the focusing NLS equation 
\begin{equation}\label{eq:focNLS}
i\partial_{t}u+\Delta u=-|u|^{p-1}u,
\end{equation}
(namely, we are ignoring the defocusing perturbation in \eqref{normNLS}) it is clear that $\tilde E(\tilde u_{\rho})=C K(\tilde u_\rho),$ for some constant $C.$ Let us recall that \eqref{eq:focNLS} is invariant under the scaling $u^{\lambda}(t,x)= \lambda^{\frac{2}{p-1}} u(\lambda^2 t,\lambda x).$
Hence, if we take $\tilde u_1,$ with normalized $L^2$-norm $\|\tilde u_1\|_{L^2}=1$, we get that $\|\tilde u_1^{\lambda}\|_{L^2}^2=\lambda^{\frac{7-3p}{p-1}}$
and $K(\tilde u_1^{\lambda})=\lambda^{\frac{5-p}{p-1}} K(\tilde u_1)$. By defining $\rho^2=\lambda^{\frac{7-3p}{p-1}},$ we deduce that $\tilde E(\tilde u_{\rho})=C\rho^{\frac{2(p-5)}{3p-7}}$ and hence $\lim_{\rho \to 0}\tilde I(\rho^2)=+\infty,$ as $p\in\left(\frac73,5\right).$ \\

\noindent With the above considerations, the claim follows provided that  inequality $I(\rho^2)\geq \tilde I(\rho^2)$ holds true. To prove the latter inequality, let us take $\bar u \in V(\rho^2)$ such that $E(\bar u)=I(\rho^2)+\varepsilon,$ $\varepsilon>0.$ Clearly, $ \tilde G(\bar u)<0.$  By considering $\tilde \mu \neq 1$ such that $\tilde G(\bar u^{\tilde \mu})=0$, by definition of the mountain pass energy we have 
\[
\tilde I(\rho^2)\leq \tilde E(u_0^{\tilde \mu})= \max_{\mu}  \tilde E(u_0^{\mu})\leq \max_{\mu}   E(u_0^{\mu})=E(u_0)=I(\rho^2)+\varepsilon.
\]
By the arbitrariness of the choice of $\varepsilon,$ we have therefore the desired inequality $I(\rho^2)\geq \tilde I(\rho^2).$ The proof is therefore concluded as $\lim_{\rho \to 0}\tilde I(\rho^2)=+\infty.$ 
\end{proof}

According to Proposition \ref{lem:growth}, let us recall that we define by $\tilde\mu=\tilde\mu(v)$ the unique scaling parameter such that $G(v^{\tilde\mu})=0$.

\begin{lemma}\label{lem:regmax}
The function $\theta \mapsto E( (\theta u)^{\tilde\mu(\theta u)})$ is $C^1$ in a neighbourhood of $\theta=1$.
\end{lemma}
\begin{proof}
Let us define the function,
\[
\begin{aligned}
f(\theta)&=\max_{\mu>0}\{E((\theta u)^\mu)\}\\
&=\max_{\mu >0} \left \{ \frac 12 \mu^2 \theta^2K(u)  + \frac{1}{q+1}\mu^{\frac 32(q-1)} \theta^{q+1}N_q(u) - \frac{1}{p+1} \mu^{\frac32 (p-1)}\theta^{p+1} N_p(u) \right \}.
\end{aligned}
\]
To simplify the notation, let us call $a=\frac 12 K(u)$, $b=\frac{1}{q+1} N_q(u)$, and  $c=   \frac{1}{p+1} N_p(u).$ Let
\[
g(\theta,\mu)= a \theta^2  \mu ^2+b \theta^{q+1} \mu^{\frac 32(q-1)} - c \theta^{p+1}  \mu^{\frac 32(p-1)},
\]
then
\[
\partial_{\mu} g(\theta, \mu)=2a \theta^2  \mu +\frac{3b}{2}(q-1) \theta^{q+1} \mu^{\frac{3}{2}(q-1)-1} - \frac{3c}{2}(p-1)\theta^{p+1}  \mu^{\frac 32(p-1)-1}
\]
and
\[
\partial_{\mu}^2 g(\theta, \mu)=2a \theta^2  - \frac{3b}{2}(q-1)\left(\frac {3q-5}{2}\right) \theta^{q+1} \mu^{\frac 32(q-1)-2}  - 
\frac{3c}{2}(p-1)\left(\frac{3p-5}{2}\right) \theta^{p+1} \mu^{\frac 32(p-1)-2}.
\]
By Proposition \ref{lem:growth}, for a $\theta_0>0$ there exists a unique $\mu_0=\mu_0(\theta_0)>0$, such that $\partial_{\mu} g(\theta_0,\mu_0)=0$ (see point $(i)$), and $\partial_{\mu\mu}^2 g(\theta_0,\mu_0)<0$ (see point $(iii)$), thus $f(\theta_0)=g(\theta_0,\mu_0)$. Then, by  applying the Implicit Function Theorem to the function $\partial_{\mu} g(\theta, \mu)$, we deduce the existence of a continuous function $\mu=\mu(\theta)$ in some neighbourhood  $\mathcal O_{\theta_0}$ of $\theta_0$ which satisfies $\partial_{\mu} g(\theta,\mu(\theta))=0$ for any $\theta\in\mathcal O_{\theta_0},$ and $\partial_{\mu\mu}^2 g(\theta,\mu(\theta))<0$.  Now, since the function $g(\theta, \mu)$ is $C^1$ in $(\theta, \mu)$, it follows that $f(\theta)$ is $C^1$.
\end{proof}

\begin{lemma} \label{lem:regmax2} For $p \in \left(\frac{7}{3},5\right)$, $1<q<p$, and $a>0$, $b>0,$ $c>0,$ the function  $(a,b,c)\mapsto f(a,b,c)$ where 
\[
f(a,b,c)=\max_{z>0} \left \{ az^2+bz^{\frac32(q-1)} - cz^{\frac{3}{2}(p-1)} \right \},
\]
is continuous on $\mathbb{R}^+\times \mathbb{R}^+ \times \mathbb{R}^+$.
\end{lemma}
\begin{proof}
The proof is similar to the one of Lemma \ref{lem:regmax2}.
\end{proof}
We can eventually prove the weak monotonicity of $I(\rho^2)$. 
\begin{lemma}\label{lem:nonincreasing}
The function $\rho \mapsto  I(\rho^2)$ is continuous in $(0, \infty)$ and nonincreasing.
\end{lemma}
\begin{proof}
The continuity of  $\rho \mapsto  I(\rho^2)$ is a standard fact. To show that  $\rho \mapsto I(\rho^2)$ is nonincreasing, we follow the approach in \cite[Lemma 5.3]{BJL}. It is enough to verify that, for any $\rho_1<\rho_2$ and $\varepsilon >0$ arbitrary, we have
\[
I(\rho_2^2)\leq I(\rho_1^2)+\varepsilon.
\]
By definition of $I(\rho_1^2)$, there exists $u_1 \in V(\rho_1^2)$ such that $E(u_1)\leq  I(\rho_1^2)+ \frac{\varepsilon}{2}$
with
\begin{equation*}\label{444}
E(u_1)=\max_{\mu>0}E(u_1^{\mu}).
\end{equation*}
We truncate $u_1$ with a compactly supported function as follows. Let $\eta \in C_0^{\infty}(\mathbb{R}^3)$ be a radial cut-off function such that $\eta:\R^3\mapsto[0,1]$
\begin{equation}\label{eq:cutrad}
\eta (x)=\begin{cases}
1, &  \hbox{ for } \quad |x|\leq 1\\
0, & \hbox{ for } \quad  |x |\geq 2
\end{cases}.
\end{equation}
For any  $\delta>0$, let
\[
\widetilde{u}_{1,\delta}(x)=\eta(\delta x) u_1(x).
\]
It is standard to show that
$\widetilde{u}_{1,\delta}(x) \to u_1(x)$  in $ H^1$ as $ \delta \to 0.$ Then, by continuity, we have, as $\delta \to 0$,
\begin{equation*}\label{avv}
\begin{aligned}
K(\widetilde{u}_{1,\delta})  &\to K(u_1),\\
N_{q}(\widetilde{u}_{1,\delta}) &\to N_q(u_1), \\
N_{p}(\widetilde{u}_{1,\delta}) & \to N_p(u_1). 
\end{aligned}
\end{equation*}
At this point, using Lemma \ref{lem:regmax2}, we deduce that there exists $\delta >0$ small enough, such that
\begin{equation*}\label{77}
\max_{\mu>0}E(\widetilde{u}_{1,\delta}^\mu)= \max_{\mu>0} E(u_1^\mu)   + \frac{\varepsilon}{4}.
\end{equation*}
Now let $v(x)\in C_0^{\infty}(\mathbb{R}^3)$ be radial and such that $\hbox{supp}\,v \subset B_{2R_{\delta}+1}\backslash B_{2R_{\delta}}$, where we defined $R_\delta=\frac2\delta.$ Then we introduce the function
\[ 
v_0=\frac{\rho_2^2- \|\widetilde{u}_{1,\delta}\|_{L^2}^2}{\|v\|_{L^2}^2}\cdot v,
\]
for which we have $\|v_0\|_{L^2}^2=\rho^2_2- \|\widetilde{u}_{1,\delta}\|_{L^2}^2$. Finally, using the scaling \eqref{eq:scaling}, letting $v_0^{\mu}=\mu^{\frac{3}{2}}v_0(\mu x)$, for $\mu \in (0,1)$, we have $\|v_0^{\mu}\|_{L^2}^2=\|v_0\|_{L^2}^2$ and the rescaled quantities
\begin{equation}\label{eq7}
\begin{aligned}
K( v_0^{\mu} )&= K( v_0 ) \mu^2,\\
N_q(v_0^{\mu})  &= N_q(v_0)\mu^{\frac 32(q-1)} ,\\
N_p(v_0^{\lambda} ) &= N_p(v_0 )\mu^{\frac{3}{2}(p-1)} .
\end{aligned}
\end{equation}
Now, for any $\mu \in (0,1)$ we define $w_{\mu}=\widetilde{u}_{1,\delta} + v_0^{\mu}$. We observe that
\begin{equation*}\label{suport}
\hbox{dist}\{\hbox{supp}\, \widetilde{u}_{1,\delta}, \hbox{supp}\, v_0^{\mu} \} \geq \frac{2R_{\delta}}{\mu}-R_{\delta}=\frac{2}{\delta}\left(\frac{2}{\mu}-1\right)>0.
\end{equation*}
It follows that $\|w_{\mu}\|_2^2=\|\widetilde{u}_{1,\delta}\|_2^2+\|v_0^{\mu}\|_2^2$ and $w_{\mu} \in S(\rho_2^2)$. Furthermore, 
\begin{equation}\label{eq11}
\begin{aligned}
K( w_{\mu})&=K(\widetilde{u}_{1,\delta})+K(v_0^{\mu}), \\
N_q(w_{\mu})&=N_q(\widetilde{u}_{1,\delta}) + N_q (v_0^{\mu}),\\
N_p(w_{\mu})&=N_p(\widetilde{u}_{1,\delta}) + N_p(v_0^{\mu}).
\end{aligned}
\end{equation}
From \eqref{eq11} and \eqref{eq7} we see that, as $\mu\to0,$
\begin{equation*}\label{14}
\begin{aligned}
K(w_{\mu}) &\to K(\widetilde{u}_{1,\delta} ),  \\
N_q(w_{\mu}) &\to N_q(\widetilde{u}_{1,\delta}) \\
N_p(w_{\mu})& \to N_p(\widetilde{u}_{1,\delta} ).
\end{aligned}
\end{equation*}
Thus,  we have that, by fixing $\mu>0$ small enough,
\begin{equation*}\label{15}
\max_{s>0}E(w_{\mu}^s) \leq \max_{s>0} E(\widetilde{u}_{1,\delta}^s)+\frac{\varepsilon}{4}.
\end{equation*}
At this point we can conclude  with the following:
\begin{equation*}
I(\rho_2^2)\leq \max_{s>0}E(w_{\mu}^s)\leq \max_{s>0} E(\widetilde{u}_{1,\delta}^s)+\frac{\varepsilon}{4} 
\leq \max_{s>0} E(u_1^s) + \frac{\varepsilon}{2} 
= E(u_1) + \frac{\varepsilon}{2}  \leq I(\rho_1^2)+\varepsilon,
\end{equation*}
and this ends the proof.
\end{proof}

We are able to prove the result concerning the connection between  the  function $I(\rho^2)$ and the existence of solutions to \eqref{EE}.
\begin{lemma}\label{lem:necessary}
If the   function $\rho \mapsto   I(\rho^2)$ is not a strictly monotone decreasing  function in $(0, \rho]$, then there exist  $\rho_0\in (0, \rho)$ and $u_0\in S(\rho_0^2)$ such that $u_0$ is a positive radially symmetric solution  to \eqref{EE}.
\end{lemma}
\begin{proof}
If the   function $\rho \mapsto   I(\rho^2)$ is not a  strictly monotone decreasing  function in $(0, \rho]$ there exists $\bar \rho \in (0, \rho)$ such that, given the quantities $\Lambda$ and $\lambda$ defined by
\[
\Lambda:=\inf_{\rho \in (0, \bar \rho)} I(\rho^2)
\]
and 
\[
\lambda=\inf \left\{\rho \in (0, \bar \rho] \quad\hbox{ such that }\quad   I(\rho^2)=\Lambda \right\},
\]
then we have $\lambda<\bar \rho$. Notice that $\lim_{\rho \to 0}   I(\rho^2)=+\infty$  and hence $\lambda>0$.\\
By construction, $I(\lambda^2)<I(s^2)$
for any $0<s< \lambda,$ and the ground state at the mountain pass energy level for the constrained problem with mass $\lambda$ is achieved by Lemma \ref{lem:mono}. Moreover, this ground state can be chosen  positive and radially symmetric. Let us call $u_0$
such ground state with mass $\lambda$. We have that the function $ I(\rho^2)$ has a global minimum in $\lambda$ in the interval $(0, \bar \rho )$ and thus
\[
E(u_0)=  I(\lambda^2)\leq    I(\theta^2\lambda^2)\leq  E\left((\theta u)^{\tilde\mu(\theta u_0)}\right)
\]
for $\theta \in (1-\epsilon, 1+\epsilon)$. We deduce that
\[
\frac{d}{d \theta} \left( E((\theta u_0)^{\tilde\mu(\theta u_0))} \right)_{|_{\theta=1}}=0,
\]
which  in turn gives
\[
\frac{d}{d \theta}\left(\frac{(\tilde\mu(\theta u_0))^2}{2}\theta^2K(u_0) +\frac{(\tilde\mu(\theta u_0))^{\frac 32(q-1)}}{q+1}\theta^{q+1}N_q(u_0)- \frac{(\tilde\mu(\theta u_0))^{\frac 32(q-1)}}{p+1}\theta^{p+1} N_p(u_0)\right)_{|_{\theta=1}}=0.
\]
Direct computations give, using that $\tilde\mu (\theta u)$ is $C^1$ and  $\tilde\mu(\theta u_0)_{|_{\theta=1}}=1$, 
\[
\begin{aligned}
\frac{d}{d \theta} \left( E((\theta u_0)^{\tilde\mu(\theta u_0))} \right)_{|_{\theta=1}}&=\left( K(u_0)+\frac{3}{2}\left(\frac{q-1}{q+1}\right)N_q(u_0)-\frac{3}{2}\left(\frac{p-1}{p+1}\right)N_p(u_0)\right)\frac{d}{d \theta} \tilde\mu(\theta u_0)_{|_{\theta=1}}\\
&+K(u_0)+N_q(u_0)-N_p(u_0).
\end{aligned}
\]
By recalling that $K(u_0)+\frac{3}{2}\left(\frac{q-1}{q+1}\right)N_q(u_0)-\frac{3}{2}\left(\frac{p-1}{p+1}\right)N_p(u_0)=0$ as $G(u_0)=0$,   we get
\begin{equation}\label{eq:0lag}
K(u_0)+N_q(u_0)-N_p(u_0)=0.
\end{equation}
Now, $u_0$ solves
\[
-\Delta u_0+\omega u_0 +|u_0|^{p-1}u_0-|u_0|^{q-1}u_0=0
\]
where $\omega$ is the associated Lagrange multiplier.
From \eqref{eq:0lag} we obtain that $\omega=0$ and hence that $u_0$ solves the zero mass case equation \eqref{EE}. It is therefore enough to set $\rho_0=\lambda$ to conclude the proof of the Lemma. 
\end{proof}

\section{Proof of Theorem \ref{theorem:main1}}
This section is concerned with a proof of the main Theorem  \ref{theorem:main1}. Let us consider the zero-mass elliptic equation \eqref{EE}, which we rewrite:
\begin{equation*}
-\Delta u+|u|^{q-1}u-|u|^{p-1}u=0.
\end{equation*}
By Berestycki-Lions \cite{BL}, we know that it
admits a real, nonnegative, radial, and nonincreasing solution $U\in X=\dot H^{1}\cap L^{q+1}.$ In \cite{BL}, the ground state for the  equation above is found as a minimizer of the constrained minimization problem
\begin{equation}\label{eq:minBL}
\inf \left\{ \frac 12 K(u) \,\hbox{ such that }\, u\in X, \hbox{ and } \frac{1}{p+1}N_p(u)- \frac{1}{q+1}N_q(u)=1\right\}.
\end{equation}
The minimizer, that can be chosen positive and radially symmetric, solves
\[
-\Delta u=\lambda(u^p-u^q)
\]
for some $\lambda>0,$ which can be normalized by a suitable rescaling, in order to get a solution to \eqref{EE}. Moreover, by \cite{KZ,LN}, $U$ is unique. A priori, $U\notin L^2.$ Nonetheless, in \cite{DS,DSW}, the precise decay of the solution $U$  was proved.
In particular 
\begin{equation*}
|U(r)|\sim \frac{1}{r^{\alpha}} \quad\hbox{ with }\quad \alpha=\max\left\{\frac{2}{q-1}, 1\right\},
\end{equation*}
and hence $U\in L^2$ if and only if $1<q<\frac73$, i.e.  the defocusing nonlinearity is mass-subcritical.\\
\\
On the other hand, the functional $E$ has a mountain pass geometry in $X,$ and called $I$ the mountain pass energy, and $u_1$ the function in $X$ that minimizes
\eqref{eq:minBL}, we have that

\begin{equation}\label{I-BV}
\begin{aligned}
I=\max_{s>0}E\left(u_1\left(\frac{x}{s^{1/3}}\right)\right)&=\max_{s>0}\left\{\frac{1}{2}K(u_1)s^{\frac{1}{3}}-s\left(\frac{1}{p+1}N_p(u_1)- \frac{1}{q+1}N_q(u_1)\right)\right\}\\
&=\max_{s>0}\left\{\frac{1}{2}K(u_1)s^{\frac{1}{3}}-s\right\}.
\end{aligned}
\end{equation}
This connection between the mountain pass energy and the constrained minimization of \cite{BL} has been proved, in a general framework, in \cite[Theorem 0.1]{BV}.\\
In particular, $\max_{s>0}\left\{\frac{1}{2}K(u_1)s^{\frac{1}{3}}-t\right\}$ can be computed explicitly,  and we get that 
\begin{equation*}\label{eq:bevi}
I=\left(\frac{1}{2}\frac{1}{6^{1/2}}-\frac{1}{6^{3/2}}\right)K(u_1)^{\frac 32}=\frac{2}{6^{3/2}}K(u_1)^{\frac 32}.
\end{equation*}
By calling $s_0=\left(\frac{K(u_1)}{6}\right)^{\frac 32}$ the maximum of the function $h(s)=\frac{1}{2}K(u_1)s^{1/3}-s,$ we have that $U=u_1\left(\frac{x}{s_0^{1/3}}\right)$, and hence from \eqref{I-BV} we also obtain that the mountain pass energy is given $I=E(U)=E\left(u_1(\cdot/s_0^{1/3})\right)$.\\

\noindent\textbf{Mass-subcritical perturbations.} By Proposition \ref{prop:mp}, $E(u)$ has a mountain pass geometry on $S(\rho^2)$. By Lemma \ref{lem:mono} and Lemma \ref{lem:necessary},
when $1<q<\frac73,$ the function $\rho \mapsto I(\rho^2)$ is strictly decreasing for  $\rho \in (0, \rho_c]$, where $\rho_c^2=\|U\|_{L^2}^2$, and a positive and radially symmetric
ground state $u_{\rho}$ exists for any $\rho$ in that interval.\\

When $\rho=\rho_c$, we claim that $U=u_{\rho_c}.$ Indeed
\begin{equation*}\label{eq:abc}
I =  E(U)=\inf_{W} E=\inf_{V(\rho_c^2)} E=I(\rho_c^2)=E(u_{\rho_c})\leq I(\theta^2\rho_c^2)\leq E\left((\theta u_{\rho_c})^{\tilde\mu(\theta u_{\rho_c})}\right)
\end{equation*}
for $\theta \in (1-\epsilon, 1+\epsilon),$ and we recall that the space  $W$ was defined in \eqref{eq:space-W} by
\[
W=\left\{u\in X  \,\hbox{ such that } \, G(u)=0\right\}.
\]
Arguing as in Lemma \ref{lem:necessary} we get 
\[
\frac{d}{d \theta} \left( 
E \left( 
(\theta u_{\rho_c})^{\tilde\mu(\theta u_{\rho_c})}\right) 
\right)_{\vert_{\theta=1}}=0,
\]
and hence $u_{\rho}$ is a positive solution of $\eqref{EE}$, i.e. $U=u_{\rho_c}.$\\

By the fact that $I(\rho_c^2)=I= \displaystyle \inf_{W} E\leq \displaystyle \inf_{V(\rho^2)} E$ for any $\rho$ and that Lemma \ref{lem:nonincreasing} holds true, then $I(\rho^2)$ is constant for $\rho\geq \rho_c$.
If $I(\rho^2)$ were achieved for some $\rho>\rho_c,$ arguing again as in Lemma \ref{lem:necessary} we would find again
$\frac{d}{d \theta} \left( E((\theta u_{\rho})^{\tilde\mu(\theta u_{\rho}))} \right)_{|_{\theta=1}}=0$
by the fact that $I=I(\rho^2)\leq I(\theta^2\rho ^2)\leq E((\theta u_{\rho})^{\tilde\mu(\theta u_{\rho})})$.
Moreover, we  would find a second positive radially symmetric solution to \eqref{EE} which is different from $U,$ and this is forbidden by the uniqueness result.\\

\noindent\textbf{Mass-critical or mass-supercritical perturbations.} By Lemma \ref{lem:mono} and Lemma \ref{lem:necessary},
when $\frac73\leq q<p$ the function $\rho \mapsto I(\rho^2)$ is strictly decreasing for  $\rho \in (0, \infty),$ because  positive, radially symmetric solutions to \eqref{EE} do not belong to $L^2,$ and hence  a positive, radially symmetric
ground state $u_{\rho}$ exists for any $\rho$ in the  positive half-line.
\\
To prove that $\lim_{\rho\to \infty}I(\rho^2)=E(U),$ we argue as in Lemma \ref{lem:nonincreasing}, by considering  a radial cut-off function $\eta \in C_0^{\infty}(\mathbb{R}^3;[0,1])$ such that  \eqref{eq:cutrad} holds.
For any small $\delta>0$, let
\[
U_\delta(x)=\eta(\delta x) U(x).
\]
Clearly, $\|U_{\delta}\|_{L^2}^2=\rho_{\delta}^2\to \infty$ when $\delta\to 0$ and
\begin{equation*}\label{avv2}
\begin{aligned}
K(U_{\delta})  &\to K(U),\\
N_{q}(U_{\delta})& \to N_q(U),\\
N_{p}(U_{\delta})  &\to N_p(U). 
\end{aligned}
\end{equation*}
As in Lemma  \ref{lem:nonincreasing}, for a fixed $\varepsilon>0$ there exists $\delta>0$ small enough such that
\[
I(\rho_{\delta}^2)\leq \max_{\mu>0}E(U_{\delta}^{\mu})\leq \max_{\mu >0} E(U)+\varepsilon=I+\varepsilon,
\]
and hence $\lim_{\rho\to \infty}I(\rho^2)=E(U)$.\\

The proof of Theorem \ref{theorem:main1} is concluded.


\section{Blow-up}
This section is devoted to the proof of the  blow-up results for solutions to \eqref{NLS}.  In order to prove Theorem \ref{thm:blowup},  we introduce some fundamental estimates.  
\begin{lemma}\label{lem:2.2}
Let us suppose that the initial datum $u_0$ satisfies  $E(u_0)<I(\|u_0\|_{L^2}^2)$ and $G(u_0)<0.$ Then $G(u(t))<0$ for any $t\in(-T_{min}, T_{max}).$ More precisely, there exists a positive constant $\delta>0$ such that $G(u(t))\leq-\delta$ for any $t\in(-T_{min}, T_{max}).$ 
\end{lemma}
\begin{proof}
By the absurd, if $G(u(t))>0$ for some time $t\in(-T_{min}, T_{max}),$ then by the continuity in time of the function $G(u(t)),$ there exists $\tilde t$ such that $G(u(\tilde t))=0.$ By definition of the functional $I(\rho),$ we have therefore $I(\|u_0\|_{L^2}^2)\leq E(u(\tilde t))=E(u_0),$ which is a contradiction with respect to the hypothesis.

The uniform lower bound away from zero is proved as follows. We simply write $u$ for the time-dependent solution $u(t).$ By Proposition \ref{lem:growth} there exists a scaling parameter  $\tilde\mu\in(0,1)$ such that $G(u^{\tilde\mu})=0.$ Then 
\[
E(u_0)-E(u^{\tilde\mu})=(1-\tilde\mu)\frac{d}{d\mu}E(u^\mu)\vert_{\mu=\bar\mu}
\] 
for some $\bar\mu\in(\tilde\mu,1),$ and due to the concavity of $\mu\mapsto E(u^\mu)$ -- see again Proposition \ref{lem:growth} -- we have that
\[
E(u_0)-E(u^{\tilde\mu})=(1-\tilde\mu)\frac{d}{d\mu}E(u^\mu)\vert_{\mu=\bar\mu}\geq (1-\tilde\mu) \frac{d}{d\mu}E(u^\mu)\vert_{\mu=1}=(1-\tilde\mu)G(u).
\] 
The last equality in the above chain of relations comes from the last statement in Proposition \ref{lem:growth}. Hence 
\[
G(u)\leq (1-\tilde\mu)^{-1}( E(u_0)-E(u^{\tilde\mu})).
\]
As $G(u^{\tilde\mu})=0,$ we have that $I(\|u_0\|_{L^2}^2)\leq E(u^{\tilde\mu})$ by definition, thus 
\[
G(u(t))\leq -(1-\tilde\mu)^{-1} (I(\|u_0\|_{L^2}^2)-E(u_0)).
\]
The proof is complete by choosing $\delta:=-(1-\tilde\mu)^{-1}( I(\|u_0\|_{L^2}^2)-E(u_0)).$
\end{proof}

As a consequence of Lemma \ref{lem:2.2}, we can actually provide a  pointwise-in-time upper bound for the Pohozaev functional $G$ evaluated at $u(t).$ 
\begin{lemma}\label{lem:2.3}
Under the hypothesis of Lemma \ref{lem:2.2}, there exists a strictly positive constant $\epsilon>0,$ such that $G(u(t))\leq-\epsilon \|u(t)\|_{\dot H^1}^2$ for any $t\in(-T_{min},T_{max}).$
\end{lemma}
\begin{proof}
Firstly, by recalling the identity \eqref{eq:egkn}, we write
\[
K=\frac{6(p-1)}{3p-7}\left( E-\frac{2}{3(p-1)}G-\frac{p-q}{(q+1)(p-1)}N_q\right).
\]
It straightforwardly follows that 
\[
\begin{aligned}
G+\epsilon K&=\left(1-\frac{4\epsilon}{3p-7}\right)G+\frac{6\epsilon(p-1)}{3p-7}E-\frac{6\epsilon(p-q)}{(3p-7)(q+1)}N_q\\ 
&\leq \left(1-\frac{4\epsilon}{3p-7}\right)G+\frac{6\epsilon(p-1)}{3p-7}E,
\end{aligned}
\]
as the last term in the right-hand side in the first line is negative. Provided $\epsilon\ll1$ we get therefore 
\[
G+\epsilon K\leq (1-\varepsilon)G+\varepsilon^\prime E
\]
for some $\varepsilon, \varepsilon^\prime\ll1.$ Hence, by using Lemma \ref{lem:2.2}, we obtain (possibly by refining the choice of $\varepsilon$)
\[
G+\epsilon K\leq (1-\varepsilon)\delta+\varepsilon^\prime E\leq0.
\]
\end{proof}
A corollary of the of the previous Lemma is a uniform lower bound for the $\dot H^1$-norm of the solution to \eqref{NLS}.
\begin{corollary}\label{coro1}
Under the hypothesis of Lemma \ref{lem:2.2}, there exists a positive constant $c>0$ such that 
\[
\inf_{t\in(-T_{min},T_{max})}\|u(t)\|_{\dot H^1}\geq c.
\]
\end{corollary}
\begin{proof}
The existence of a sequence of times $\{t_n\}_{n}\subset \R$ such that $\lim_{n\to\infty}\|u(t_n)\|_{\dot H^1}=0,$ would imply that  $\lim_{n\to\infty}\left(\|u(t_n)\|_{L^{p+1}}+\|u(t_n)\|_{L^{q+1}}\right)=0$ as well, by employing the Gagliardo-Nirenberg's inequality. But then $G(u(t_n))\to0,$ which is in contrast with respect to what Lemma \ref{lem:2.2} claims.
\end{proof}

We can now give a proof of the blow-up results as stated in Theorem \eqref{thm:blowup}. In order to employ a convexity argument, \cite{Gla, OT}, we introduce the following virial functional.
Define 
\[
V_\phi(t)=\int \phi|u(t)|^2\,dx.
\]
 where $\phi$ is a time-independent smooth cut-off function.  Usual calculations (see \cite{TC}) imply,  by using the equation solved by $u(t),$
\[
V_\phi^\prime(t)=2\Im \left\{\int\nabla\phi \bar u\nabla u\,dx\right\}
\]
and
\begin{equation*}
\begin{aligned}
V_\phi^{\prime\prime}(t)&=4\Re \left\{\int\nabla^2\phi  \nabla u\nabla\bar u\,dx\right\}-\int\Delta^2 \phi |u|^2\,dx\\
&+2\left(1-\frac{2}{q+1}\right)\int\Delta\phi |u|^{q+1}\,dx-2\left(1-\frac{2}{p+1}\right)\int\Delta\phi |u|^{p+1}\,dx\\
&=4\Re \left\{\int\nabla^2\phi  \nabla u\nabla\bar u\,dx\right\}-\int\Delta^2 \phi |u|^2\,dx\\
&+2\left(\frac{q-1}{q+1}\right)\int\Delta\phi |u|^{q+1}\,dx-2\left(\frac{p-1}{p+1}\right)\int\Delta\phi |u|^{p+1}\,dx\\
&=4\Re \left\{\int\nabla^2\phi  \nabla u\nabla\bar u\,dx\right\}-\int\Delta^2 \phi |u|^2\,dx+\int \Delta \phi n_{p,q}(u)dx,
\end{aligned}
\end{equation*}
where we have introduced the notation 
\[
n_{p,q}(u)=\frac{2(q-1)}{q+1}|u|^{q+1}-\frac{2(p-1)}{p+1}|u|^{p+1}
\]
to denote the density (up to some constants) of the potential energy. To lighten the notation, we omitted the dependence on time of $u.$ 
\begin{proposition} \label{rem-viri-iden}
Provided the cut-off function is smooth enough, we have the following identities:
\begin{itemize}
		\item[(i)] If $\phi(x) = |x|^2$, 	
		\begin{equation}\label{eq:variance}
		V_\phi^{\prime\prime}(t) = 8G(u(t)).
		\end{equation}
		\item[(ii)] If $\phi$ is radially symmetric, by denoting $|x|=r,$ we have		
		\begin{equation*}\label{cor:ii}
		\begin{aligned}
		V_\phi^{\prime\prime}(t) &= -\int \Delta^2 \phi(x) |u|^2 dx + 4\int \frac{\phi'(r)}{r} |\nabla u|^2 dx \\
		& + 4 \int \left(\frac{\phi''(r)}{r^2} - \frac{\phi'(r)}{r^3} \right) |x\cdot \nabla u|^2 dx \\
		&+\int \Delta \phi(x)n_{p,q}(u)dx. 
		\end{aligned}
		\end{equation*}
		\item[(iii)] If $\phi$ is radial and $u$ is  radial as well, then 
		\begin{equation}\label{eq:vir-rad}
		V_\phi^{\prime\prime}(t) = -\int \Delta^2 \phi(x)|u|^2 dx  + 4 \int \phi''(r) |\nabla u|^2 dx +\int \Delta \phi(x)n_{p,q}(u) dx
\end{equation}

\item[(iv)]  Let $\psi: \R^2 \to \R$  a smooth radial function. By setting $\phi(x) = \psi(\bar x) + x_3^2,$ provided $u(t) \in \Sigma_3 $ for all $t\in (-T_-,T_+)$, then we have
		\begin{equation}\label{eq:vir-cyl}
		\begin{aligned}
		 V_\phi^{\prime\prime}(t)&= -\int \Delta^2_{\bar x} \psi(\bar x) |u|^2 dx + 4\int \psi''(\rho) |\nabla_{\bar x} u|^2  dx \\
		& + 8 \|\partial_{x_3} u\|^2_{L^2}+\int (2+\Delta_{\bar x}\psi(\bar x)) n_{p,q}(u)dx,
		\end{aligned}
		\end{equation}
		where $\rho = |\bar x|.$ 
	\end{itemize}
\end{proposition}
\noindent We are now able to give a proof of the blow-up results of Theorem \ref{thm:blowup}.

\begin{proof}[Proof of Theorem \ref{thm:blowup}] By using the above virial identities, we prove $(i),(ii),(iii)$ in order. \\
\noindent \emph{Proof of $(i)$.} The first point is proved by employing a usual Glassey-type argument, once we have that $G(u(t))\leq-\delta<0$ for any $t\in(-T_{min},T_{max}),$ see Lemma \ref{lem:2.2}. Indeed, by \eqref{eq:variance} we have 
\[
V_{|x|^2}^{\prime\prime}(t)\leq-\delta,
\]
and a convexity argument gives the desired conclusion.\\

\noindent \emph{Proof of $(ii).$} For the second point, concerning radial solutions, let $\theta: [0,\infty) \to [0,2]$ be a $C^\infty_c$ function satisfying
\[
\theta(r)= \left\{
\begin{array}{ccl}
2 &\text{if}& 0 \leq r \leq 1, \\
0 &\text{if}& r\geq 2.
\end{array}
\right.
\]
We define the function $\Theta:[0,\infty) \to \R^+\cup\{0\}$ by
\begin{equation*} \label{defi-vartheta}
\Theta(r):= \int_0^r \int_0^\tau \theta(s) ds d\tau.
\end{equation*}
For $R>0$, we define the radial function $\varphi_R: \R^3 \to \R$ by means of a rescaling of $\Theta:$
\begin{equation*} \label{defi-varphi-R}
\phi_R(x)=\phi_R(r):= R^2 \Theta(r/R), \quad r=|x|.
\end{equation*}
It is straightforward to check that $\forall x \in \R^3$ and $\forall r\geq 0$,
\[
2\geq \phi''_R(r) \geq 0, \qquad 2-\frac{\phi'_R(r)}{r} \geq 0, \qquad 6-\Delta \phi_R(x) \geq 0.
\]
With the above properties and \eqref{eq:vir-rad}, we get 
	\begin{equation} \label{viri-est-rad}
	V_{\phi_R}^{\prime\prime}(t) \lesssim -1
	\end{equation}	
Indeed, by \eqref{eq:vir-rad}, and using that 
\begin{equation}\label{ide8}
8\left( G(u)-K(u)-\frac{3}{4}\int n_{p,q}(u)dx \right)=0,
\end{equation}
see \eqref{eq:G}, we obtain 	
\begin{align*}
V_{\phi_R}^{\prime\prime}(t)&= 8 G(u) - 8 K(u) - 6\int n_{p,q}(u)dx \\
&-\int \Delta^2 \phi_R(x)|u|^2 dx + 4 \int \phi''_R(r) |\nabla u|^2 dx \\
& +\int \Delta \phi_R(x) n_{p,q}(u)dx \\
&= 8 G(u) -\int \Delta^2 \phi_R(x) |u|^2  dx \\
& - 4 \int (2-\phi''_R(r))|\nabla u|^2 dx  -\int (6-\Delta \phi_R(x)) n_{p,q}(u)dx.
\end{align*}
As $\|\Delta^2 \phi_R\|_{L^\infty} \lesssim R^{-2}$, the conservation of mass implies that
\[
\left| \int \Delta^2 \phi_R(x)|u|^2 dx \right| \lesssim R^{-2}.
\]
The latter, together with $\phi''_R (r) \leq 2$ for all $r\geq0$, $\|\Delta \phi_R\|_{L^\infty} \lesssim 1$, and $\phi_R(x) = |x|^2$ on $|x| \leq R$,  yield
\[
V_{\phi_R}^{\prime\prime}(t) \leq 8 G(u) + CR^{-2} + C\int_{|x|\geq R}\left(|u|^{q+1}+|u|^{p+1}\right)dx.
\]
To estimate the last term, we recall the following radial Sobolev embedding (see e.g., \cite{CO}): for a radial function $f\in H^1(\R^3)$, we have  
\begin{align} \label{rad-sobo}
\sup_{x \ne 0} |x| |f(x)| \leq C\|\nabla f\|^{\frac{1}{2}}_{L^2} \|f\|^{\frac{1}{2}}_{L^2}.
\end{align}
Thanks to \eqref{rad-sobo} and the conservation of mass, we estimate
\begin{align*}
\int_{|x|\geq R} |u(x)|^{q+1}dx &=\int_{|x|\geq R} |u|^2|u|^{q-1}dx\lesssim \left( R^{-1}\|\nabla u\|_{L^2}^{1/2}\|u\|_{L^2}^{1/2}\right)^{q-1} \|u\|_{L^2}^2\\
&\lesssim R^{-(q-1)}\|\nabla u\|_{L^2}^{(q-1)/2},
\end{align*}
and similarly 
\begin{align*}
\int_{|x|\geq R} |u(x)|^{p+1}dx &=\int |u|^2|u|^{p-1}dx\lesssim \left( R^{-1}\|\nabla u\|_{L^2}^{1/2}\|u\|_{L^2}^{1/2}\right)^{p-1} \|u\|_{L^2}^2\\
&\lesssim R^{-(p-1)}\|\nabla u\|_{L^2}^{(p-1)/2}.
\end{align*}
It follows that
\[
V_{\phi_R}^{\prime\prime}(t) \leq 8 G(u) + CR^{-2} +CR^{-(q-1)}\left(\|\nabla u(t)\|^{(p-1)/2}_{L^2}+\|\nabla u(t)\|^{(q-1)/2}_{L^2}\right),
\]
and by using the Young inequality, observing that $\frac{q-1}{2}<\frac{p-1}{2}<2,$ we can infer that 
\[
V_{\phi_R}^{\prime\prime}(t) \leq 8 G(u(t)) + o_R(1) +o_R(1)\|\nabla u(t)\|^2_{L^2}
\]
Hence, \eqref{viri-est-rad} follows by combining the last estimate above with Lemma \ref{lem:2.2}, Lemma \ref{lem:2.3}, and Corollary \ref{coro1}.
\\

\noindent \emph{Proof of $(iii)$.} For last third point, which deals with cylindrically symmetric solutions, a refined  argument of the proof for $(ii)$ is used. It is worth mentioning that the first early work for the classical focusing NLS equation in anisotropic spaces goes back to Martel \cite{Mar}. See \cite{ADF, BF20, DF} for recent results for other classes of dispersive equations.

From \eqref{eq:vir-cyl}, and  using again the identity \eqref{ide8},
 \begin{align*}
	V_{\phi_R}^{\prime\prime}(t)  &= -\int \Delta^2_{\bar x} \psi_R(\bar x) |u|^2 dx + 4\int \psi''_R(\rho) |\nabla_{\bar x} u|^2 dx \\
	&+ 8 \|\partial_{x_3} u\|^2_{L^2} +\int(2+ \Delta_{\bar x} \psi_R(\bar x)) n_{p,q}(u)dx,\\
	&=-\int \Delta^2_{\bar x} \psi_R(\bar x) |u|^2 dx - 4\int (2-\psi''_R(\rho)) |\nabla_{\bar x} u|^2 dx \\
	&+ 8 K(u)+ \int (2+\Delta_{\bar x} \psi_R(\bar x)) n_{p,q}(u)dx\pm 6\int n_{p,q}(u),\\
	&=8G(u)-\int \Delta^2_{\bar x} \psi_R(\bar x) |u|^2 dx - 4\int (2-\psi''_R(\rho)) |\nabla_{\bar x} u|^2 dx \\
	&-\int (4-\Delta_{\bar x} \psi_R(\bar x)) n_{p,q}(u)dx
	\end{align*}
	where $\rho=|\bar x|$. Similarly to the point $(ii),$ we introduce the function $\Theta:[0,\infty) \to \R^+ \cup\{0\}$ defined by
\begin{equation*} 
\Theta(\rho):= \int_0^\rho \int_0^\tau \theta(s) ds d\tau.
\end{equation*}
For $R>0$, we define the function $\psi_R: \R^2 \to \R$ by
\begin{equation*} \label{defi-varphi-R-bis}
\psi_R(\bar x)=\psi_R(\rho):= R^2 \Theta(\rho/R), \quad \rho=|\bar x|.
\end{equation*}
Since $\psi''_R(\rho) \leq 2$ and $\|\Delta_{\bar x}^2\psi_R\|_{L^\infty}\lesssim R^{-2},$ 
it follows that
\begin{equation}\label{eq:est-vir-cyl}
V_{\phi_R}^{\prime\prime}(t)\leq 8 G(u) +CR^{-2}-\int (4-\Delta_{\bar x}\psi_R(\bar x)) n_{p,q}(u)dx.
\end{equation}
Note that in the last estimate, the contribution given by the lower order term $|u|^{q+1}$ shall be discarded, as it accounts for a nonpositive contribution, hence  \eqref{eq:est-vir-cyl} can be refined as 
\begin{equation}\label{eq:est-vir-cyl2}
V_{\phi_R}^{\prime\prime}(t)\leq 8 G(u(t)) +CR^{-2}-C\int (4-\Delta_y\psi_R(y)) |u|^{p+1}dx.
\end{equation}

\noindent Observe that the  function
\[
h_R(\rho):=4-\Delta_{\bar x}\psi_R(\bar x)=4-\Theta^{\prime\prime}(\rho/R)-\frac R\rho\Theta^\prime(\rho/R)=4-\theta(\rho/R)-\frac R\rho\int_0^{\rho/R}\theta(s)ds
\] 
is a localization function outside a cylinder. Indeed, we note that by its construction $\Delta_{\bar x}\psi_R(\bar x)=4$ for any $\rho\leq R,$ hence $h_R(\rho)$  is a nonnegative, bounded, smooth function supported outside $\{\rho\geq R\}.$  Let us consider the last term in the right-hand side of \eqref{eq:est-vir-cyl2}:
\[
\int (4-\Delta_{\bar x}\psi_R(\bar x))|u|^{p+1}dx= \int h_R(\rho) |u|^{p+1}dx.
\]
We restrict to the range $p\in\left(\frac73,3\right].$  By using the H\"older inequality, we estimate 
\begin{equation}\label{eq:est:caze}
\begin{aligned}
\int h_R(\rho)|u|^{p+1}dx&=\int_{\R}\left(\int_{\R^2} h_R(\rho)|u|^{p-1}|u|^2d\bar x\right)dx_3\leq \int  \|h_R^{1/(p-1)}u\|_{L^\infty_{(|\bar x|\geq R)}}^{p-1}\|u\|_{L^2_{\bar x}}^2dx_3\\
&\lesssim \left( \int  \|h_R^{1/(p-1)}u\|_{L^\infty_{(|\bar x|\geq R)}}^{2}dx_3\right)^{(p-1)/2}\left( \int \|u\|_{L^2_{\bar x}}^{4/(3-p)}dx_3\right)^{(3-p)/2}.
\end{aligned}
\end{equation}

\noindent By the Gagliardo-Nirenberg inequality and the conservation of mass, we bound 
\begin{equation}\label{eq:est:L2}
\left( \int \|u\|_{L^2_{\bar x}}^{4/(3-p)}dx_3\right)^{(3-p)/2}\lesssim\|\nabla u\|_{L^2}^{(p-1)/2}.
\end{equation}
Indeed, by recalling that for a one-dimensional function  $f\in H^1(\R)$
\[
\|f\|_{\sigma}\lesssim\|f^\prime\|_{L^2}^{\theta}\|f\|_{L^2}^{1-\theta},
\]
for $\sigma\in(2,\infty)$ and $\theta=(\sigma-2)/2\sigma,$ we obtain
\[
\begin{aligned}
\left( \int \|u\|_{L^2_{\bar x}}^{4/(3-p)}dx_3\right)^{(3-p)/2}&=\|\|u\|_{L^2_{\bar x}}\|_{L_{x_3}^{4/(3-p)}}^2\lesssim\|\partial_{x_3}\|u\|_{L^2_{\bar x}}\|_{L^2_{x_3}}^{(p-1)/2}\|\|u\|_{L^2_{\bar x}}\|_{L^2_{x_3}}^{(5-p)/2}\\
&\lesssim\|\partial_{x_3} u\|_{L^2}^{(p-1)/2}\lesssim \|\nabla u\|_{L^2}^{(p-1)/2},
\end{aligned}
\]
where we used $\sigma=4/(3-p)$ and $\theta=(p-1)/4,$ and the conservation of mass in the uniform constant. The term
\[
\int  \|h_R^{1/(p-1)}u\|_{L^\infty_{(|\bar x|\geq R)}}^{2}dx_3
\]
is estimated as in our previous paper, see \cite[eq. (4.13)]{BF20}. Briefly, by using the boundedness of $h_R$ and a version of \eqref{rad-sobo} in $\R^2,$
\[
\int  \|h_R^{1/(p-1)}u\|_{L^\infty_{(|\bar x|\geq R)}}^{2}dx_3\lesssim R^{-1}\|\nabla u\|_{L^2},
\]
hence, 
\begin{equation}\label{eq:grad:loc}
\left(\int  \|h_R^{1/(p-1)}u\|_{L^\infty_{(|\bar x|\geq R)}}^{2}dx_3\right)^{(p-1)/2}\lesssim R^{-(p-1)/2}\|\nabla u\|^{(p-1)/2}_{L^2}.
\end{equation}
Thus, by glueing up together \eqref{eq:est:caze}, \eqref{eq:est:L2}, and \eqref{eq:grad:loc}, we obtain
\[
\int h_R(\rho)|u|^{p+1}dx\lesssim  R^{-(p-1)/2} \|\nabla u\|_{L^2}^{(p-1)}.
\]
The specific case $p=3$ is estimated in \cite{BF20}, so we omit the details. Note just that, instead of controlling $\left( \int \|u\|_{L^2_{\bar x}}^{4/(3-p)}dx_3\right)^{(3-p)/2},$ we have to deal with $\|\|u\|_{L^2_{\bar x}}\|_{L^\infty_{x_3}}^2.$ 
\\

\noindent The proof of the Theorem \ref{thm:blowup} is concluded. 
\end{proof}

\subsection*{Acknowledgement.} The  authors thank D. Ruiz for having pointed-out the references \cite{DS,DSW}. J.B. was partially supported by ``Problemi stazionari e di evoluzione nelle equazioni di campo nonlineari dispersive'' of GNAMPA 2020, and the project ``Dinamica di equazioni nonlineari dispersive'' by FONDAZIONE DI SARDEGNA 2016.
V.G. was supported in part by  Project 2017 ``Problemi stazionari e di evoluzione nelle equazioni di campo nonlineari'' of INDAM,
GNAMPA,
by the Institute of Mathematics and Informatics,
Bulgarian Academy of Sciences, by the Top Global University Project, Waseda University,  and by the University of Pisa, Project PRA 2018 49, and the project ``Dinamica di equazioni nonlineari dispersive'' by FONDAZIONE DI SARDEGNA 2016.

\end{document}